\documentclass[english]{article}
\usepackage[T1]{fontenc}
\usepackage[latin9]{inputenc}
\usepackage{amsmath}
\usepackage{amssymb}

\makeatletter

\providecommand{\tabularnewline}{\\}

\usepackage[T1]{fontenc}
\usepackage[latin9]{inputenc}
\usepackage{color}
\usepackage{amsmath}
\usepackage{amssymb}
\usepackage{mathrsfs}
\usepackage{babel}
\usepackage[active]{srcltx}
\usepackage[usenames,dvipsnames]{pstricks}
\usepackage{epsfig}
\usepackage{pst-grad} 
\usepackage{pst-plot} 

\makeatletter

\providecommand{\U}[1]{\protect\rule{.1in}{.1in}}
\providecommand{\U}[1]{\protect\rule{.1in}{.1in}}
\newtheorem{theorem}{Theorem}

\newtheorem{corollary}[theorem]{Corollary}

\newtheorem{lemma}[theorem]{Lemma}

\newtheorem{proposition}[theorem]{Proposition}
\newtheorem{remark}{Remark}

\newenvironment{proof}[1][Proof]{\noindent\textbf{#1.} }{\ \rule{0.5em}{0.5em}}

\makeatother

\makeatother

\usepackage{babel}

\begin{document}

\title{Maximal monotone operators with non-empty domain interior; characterizations
and continuity properties}

\author{M.D. Voisei}

\date{{}}
\maketitle
\begin{abstract}
In the context of general Banach spaces characterizations for the
maximal monotonicity of operators with non-empty domain interior as
well as stronger continuity properties for such operators are provided. 
\end{abstract}

\section{Introduction}

In the context of locally convex or Banach spaces there are few characterizations
for the maximal monotonicity of an operator (see e.g. \cite[Theorem 3.8]{MR1009594},
\cite[Theorem 2.3]{MR2207807}, and \cite[Theorem 6]{MR2594359}).
All these characterizations are based on special convex representations
associated to the operator.

In a finite-dimensional space a complete characterization for the
maximality of a monotone operator is given in \cite[Theorem 3.4]{MR2465513}
in terms of direct operator notions: the near convexity of its domain,
convexity of its values, graph closedness, and behavior at the boundary
of its domain.

In a Banach space context characterizations of maximality, similar
to those found in the finite-dimensional case, are available for full-space
or open convex domain monotone multi-functions (see \cite[Theorem\ 1.2]{MR0180884},
\cite[Lemma 2.2]{MR1191009}, \cite[Theorem 40.2, p. 155]{MR1723737},
\cite[Lemma 4.2]{MR2453098}, and \cite[Lemma 7.7, p. 104]{MR1238715}).

Every maximal monotone operator in a finite dimensional space has
a non-empty convex relative interior of its domain which is dense
in the domain and a convex domain closure (see e.g. \cite{MR0132379},
\cite[Theorems 6.2, 6.3]{MR1451876}, \cite[Theorem 12.41, p. 554]{MR1491362}).

That is why, characterizations of the maximality of a monotone operator
with non-empty relative (algebraic) domain interior in a general Banach
space, in terms of notions directly linked to the operator and similar
to those seen in the finite-dimensional case, constitute generalizations
of all fore-mentioned results and that is our primary goal.

Our secondary goal is to reveal several continuity properties with
respect to the strong$\times$weak-star topology for maximal monotone
operators which have non-empty domain interiors and are defined in
a normed barrelled space, such as, the closedness of the graph, the
upper semicontinuity, and the Cesari property.

The plan of the paper is as follows. The next section introduces the
reader to the main notions and notations used in this article. Section
3 dwells with the restrictions of an operator to affine subsets. Section
4 analyzes the finite-dimensional case and provides a new proof of
\cite[Theorem 3.4]{MR2465513}. In Section 5 the characterizations
previously seen in the finite dimensional context are extended to
arbitrary Banach spaces via a hemicontinuity condition, demiclosedness,
or representability. Section 6 deals with the continuity properties
of monotone demiclosed operators that have a non-empty domain interior.

\section{Preliminaries}

Let $(E,\mu)$ be a locally convex space and $A\subset E$. We denote
by {}``$\operatorname*{conv}A$'' the \emph{convex hull} of $A$,
{}``$\operatorname*{aff}A$'' the \emph{affine hull} of $A$, {}``$\operatorname*{lin}A$''
the \emph{linear hull} of $A$; {}``$\operatorname*{cl}_{\mu}(A)=\overline{A}^{\mu}$''
the $\mu-$\emph{closure} of $A$, {}``$\operatorname*{ri}_{\mu}A$''
the topological interior of $A$ with respect to $\operatorname*{cl}_{\mu}(\operatorname*{aff}A)$,
{}``$\operatorname*{core}A$'' the \emph{algebraic interior} of
$A$, {}``$^{i}A$'' the \emph{relative algebraic interior} of $A$,
and $^{\mu-ic}A:={}^{i}A$ if $\operatorname*{aff}A$ is $\mu-$closed
and $^{\mu-ic}A:=\emptyset$ otherwise, the relative algebraic interior
of $A$ with respect to $\operatorname*{cl}_{\mu}(\operatorname*{aff}A)$\emph{.} 

When the topology $\mu$ is implicitly understood (such is the case
when we deal with the strong topology of a normed  space) the use
of the  $\mu-$notation is avoided.

\medskip{}

For $f,g:E\rightarrow\overline{\mathbb{R}}:=\mathbb{R}\cup\{-\infty,+\infty\}$
we set $[f\leq g]:=\{x\in E\mid f(x)\leq g(x)\}$; the sets $[f=g]$,
$[f<g]$, and $[f>g]$ are defined similarly.

\medskip{}

Throughout this paper, if not otherwise explicitly mentioned, $(X,\|\cdot\|)$
is a non trivial (that is, $X\neq\{0\}$) normed  space, $X^{\ast}$
is its topological dual endowed with the weak-star topology $w^{\ast}$,
the topological dual of $(X^{\ast},w^{\ast})$ is identified with
$X$, the weak topology on $X$ is denoted by $w$, and the strong
topology on $X$ is denoted by $s$. The closed unit ball of $X$
is denoted by $B_{X}:=\{x\in X\mid\|x\|\le1\}$. The \emph{duality
product} of $X\times X^{\ast}$ is denoted by $\left\langle x,x^{\ast}\right\rangle :=x^{\ast}(x)=:c(x,x^{\ast})$,
for $x\in X$, $x^{\ast}\in X^{\ast}$.

\medskip{}

As usual, for $S\subset X$, $S^{\perp}:=\{x^{*}\in X^{*}\mid\langle x,x^{*}\rangle=0,$
for every $x\in S\}$, $\sigma_{S}(x^{*}):=\sup_{x\in S}\langle x,x^{*}\rangle$,
$x^{*}\in X^{*}$ and for $A\subset X^{*}$, $A^{\perp}:=\{x\in X\mid\langle x,x^{*}\rangle=0,$
for every $x\in A\}$, $\sigma_{A}(x)=\sup_{x^{*}\in A}\langle x,x^{*}\rangle$,
$x\in X$.

\medskip{}

To a multi\-function $T:X\rightrightarrows X^{\ast}$ we associate
its \emph{graph:} $\operatorname*{Graph}T=\{(x,x^{\ast})\in X\times X^{\ast}\mid x^{\ast}\in T(x)\}$,
\emph{inverse:} $T^{-1}:X^{*}\rightrightarrows X$, $\operatorname*{Graph}T^{-1}=\{(x^{*},x)\mid(x,x^{\ast})\in\operatorname*{Graph}T\}$,
\emph{domain}: $D(T):=\{x\in X\mid T(x)\neq\emptyset\}=\operatorname*{Pr}_{X}(\operatorname*{Graph}T)$,
and \emph{range}: $R(T):=\{x^{*}\in X^{*}\mid x^{*}\in T(x)\ \mathrm{for\ some}\ x\in X\}=\operatorname*{Pr}_{X^{*}}(\operatorname*{Graph}T)$.
Here $\operatorname*{Pr}_{X}$ and $\operatorname*{Pr}_{X^{*}}$ are
the projections of $X\times X^{*}$ onto $X$ and $X^{\ast}$, respectively.
When no confusion can occur, $T$ will be identified with $\operatorname*{Graph}T$.

\medskip

On $X$, we consider the following classes of functions and operators 
\begin{description}
\item [{$\Lambda(X)$}] the class formed by proper convex functions $f:X\rightarrow\overline{\mathbb{R}}$.
Recall that $f$ is \emph{proper} if $\mathrm{dom}\: f:=\{x\in X\mid f(x)<\infty\}$
is nonempty and $f$ does not take the value $-\infty$, 
\item [{$\Gamma_{\tau}(X)$}] the class of functions $f\in\Lambda(X)$
that are $\tau$--lower semi\emph{\-}continuous (\emph{$\tau$--}lsc
for short); when the topology is implicitly understood we use the
notation $\Gamma(X)$, 
\item [{$\mathcal{M}(X)$}] the class of non-empty monotone operators $T:X\rightrightarrows X^{*}$.
Recall that $T:X\rightrightarrows X^{*}$ is \emph{monotone} if $\left\langle x_{1}-x_{2},x_{1}^{\ast}-x_{2}^{\ast}\right\rangle \ge0$,
for all $(x_{2},x_{2}^{*}),(x_{2},x_{2}^{*})\in T$. 
\item [{$\mathscr{M}(X)$}] the class of maximal monotone operators $T:X\rightrightarrows X^{*}$.
The maximality is understood in the sense of graph inclusion as subsets
of $X\times X^{*}$. 
\end{description}
To a proper function $f:X\rightarrow\overline{\mathbb{R}}$ and a
topology $\tau$ on $X$ we associate: 
\begin{itemize}
\item the \emph{epigraph} of $f$: $\operatorname*{epi}f:=\{(x,t)\in X\times\mathbb{R}\mid f(x)\leq t\}$, 
\item the \emph{convex hull} of $f$: $\operatorname*{conv}f:X\rightarrow\overline{\mathbb{R}}$,
is the greatest convex function majorized by $f$, $(\operatorname*{conv}f)(x):=\inf\{t\in\mathbb{R}\mid(x,t)\in\operatorname*{conv}(\operatorname*{epi}f)\}$,
$x\in X$, 
\item the \emph{$\tau-$lsc convex hull} of $f$: $\operatorname*{cl}_{\tau}\operatorname*{conv}f:X\rightarrow\overline{\mathbb{R}}$,
is the greatest \emph{$\tau$--}lsc convex function majorized by $f$,
$(\operatorname*{cl}_{\tau}\operatorname*{conv}f)(x):=\inf\{t\in\mathbb{R}\mid(x,t)\in\operatorname*{cl}_{\tau}\operatorname*{conv}\operatorname*{epi}f\}$
, $x\in X$, 
\item the \emph{convex conjugate} of $f:X\rightarrow\overline{\mathbb{R}}$
with respect to the dual system $(X,X^{\ast})$: $f^{\ast}:X^{\ast}\rightarrow\overline{\mathbb{R}}$,
$f^{\ast}(x^{\ast}):=\sup\{\left\langle x,x^{\ast}\right\rangle -f(x)\mid x\in X\}$,
$x^{\ast}\in X^{\ast}$. 
\item the \emph{subdifferential} of $f$ at $x\in X$: $\partial f(x):=\{x^{\ast}\in X^{\ast}\mid\left\langle x^{\prime}-x,x^{\ast}\right\rangle +f(x)\leq f(x^{\prime}),\ \forall x^{\prime}\in X\}$
for $x\in\operatorname*{dom}f$; $\partial f(x):=\emptyset$ for $x\not\in\operatorname*{dom}f$.
Recall that $N_{C}=\partial I_{C}$ is the \emph{normal cone} of $C\subset X$,
where $I_{C}$ is the \emph{indicator} \emph{function} of $C\subset X$
defined by $I_{C}(x):=0$ for $x\in C$ and $I_{C}(x):=\infty$ for
$x\in X\setminus C$. 
\end{itemize}
Let $Z:=X\times X^{\ast}$. It is known that $(Z,s\times w^{\ast})^{\ast}=Z$
via the coupling \[
z\cdot z^{\prime}:=\left\langle x,x^{\prime\ast}\right\rangle +\left\langle x^{\prime},x^{\ast}\right\rangle ,\quad\text{for }z=(x,x^{\ast}),\ z^{\prime}=(x^{\prime},x^{\prime\ast})\in Z.\]
For a proper function $f:Z\rightarrow\overline{\mathbb{R}}$ all the
above notions are defined similarly. The conjugate of $f$ with respect
to the natural dual system $(Z,Z)$ induced by the previous coupling
is given by\[
f^{\square}:Z\rightarrow\overline{\mathbb{R}},\quad f^{\square}(z)=\sup\{z\cdot z^{\prime}-f(z^{\prime})\mid z^{\prime}\in Z\},\]
and by the biconjugate formula, $f^{\square\square}=\mathrm{cl}_{s\times w^{\ast}}\mathrm{\, conv}\, f$
whenever $f^{\square}$ or $\mathrm{cl}_{s\times w^{\ast}}\mathrm{\, conv}\, f$
is proper.

We consider the following classes of functions on $Z$: 

\vspace{-.5cm}

\begin{align*}
\mathcal{C}: & =\mathcal{C}(Z):=\{f\in\Lambda(Z)\mid f\geq c\},\\
\mathcal{R}: & =\mathcal{R}(Z):=\Gamma_{s\times w^{\ast}}(Z)\cap\mathcal{C}(Z),\\
\mathcal{D}: & =\mathcal{D}(Z):=\{f\in\mathcal{R}(Z)\mid f^{\square}\geq c\}.\end{align*}
It is known that $[f=c]\in\mathcal{M}(X)$ for every $f\in\mathcal{C}(Z)$
(see e.g. \cite[Lemma\ 3.1]{MR2453098}).

\medskip{}

To a multi-valued operator $T:X\rightrightarrows X^{\ast}$ we associate
the following functions: 
\begin{itemize}
\item the \emph{Fitzpatrick function} of $T$ (introduced in \cite{MR1009594}):
$\varphi_{T}:Z\rightarrow\overline{\mathbb{R}}$, $\varphi_{T}:=c_{T}^{\square}$,
where $c_{T}:Z\rightarrow\overline{\mathbb{R}}$, $c_{T}:=c+\iota_{T}$;
\item $\psi_{T}:=\operatorname*{cl}\,\!_{s\times w^{\ast}}\operatorname*{conv}c_{T}$
(first considered in \cite{MR2207807}); $\psi_{T}=\varphi_{T}^{\square}=c_{T}^{\square\square}$
whenever $\varphi_{T}$ or $\psi_{T}$ is proper (for example when
$T\in\mathcal{M}(X)$ (see e.g. \cite[Proposition\ 3.2]{MR2389004})).
\end{itemize}
Note that for every $T\subset Z$, $[\varphi_{T}\le c]$ describes
the set of all $z\in Z$ that are monotonically related (m.r. for
short) to $T$, that is, $c(z-a)\ge0$, for every $a\in T$.

We call a multi\-function $T:X\rightrightarrows X^{\ast}$ 
\begin{description}
\item [{\textmd{\emph{unique}}}] if $T$ admits a unique maximal monotone
extension;
\item [{\textmd{\emph{representable}}}] in $Z$ if $T=[f=c]$, for some
$f\in\mathcal{R}$; in this case $f$ is called a \emph{representative}
of $T$. We denote by $\mathcal{R}_{T}$ the class of representatives
of $T$; this notion was first considered in this form in \cite{MR2207807};
\item [{\textmd{\emph{dual-representable}}}] if $T=[f=c]$, for some $f\in\mathcal{D}$;
in this case $f$ is called a \emph{d--representative} of $T$ and
we denote by $\mathcal{D}_{T}$ the class of d-representatives of
$T$;
\item [{\textmd{\emph{NI}}}] or of \emph{negative infimum type in $Z$}
if $\varphi_{T}\geq c$ in $Z$;
\item [{\textmd{\emph{demiclosed}}}] if $\operatorname*{Graph}T$ is closed
with respect to the strong$\times$weak-star convergence of bounded
nets in $Z$, that is, if $x_{i}\rightarrow x_{0}$ strongly in $X$,
$x_{i}^{*}\rightarrow x_{0}^{*}$ weakly-star in $X^{*}$, $(x_{i}^{*})_{i}$
is (strongly) bounded in $X^{*}$, and $\{(x_{i},x_{i}^{*})\}_{i}\subset T$
then $(x_{0},x_{0}^{*})\in T$. 
\item [{$strongly\times weakly-star\ upper\ \mathscr{H}-semicontinuous$}] ($s\times w^{*}-$usc
for short) \emph{at $x\in X$ }if for every weak-star open set $V\supset Tx$
there exists a neighborhood $U$ of $x$ such that $T(U):=\cup_{u\in U}Tu\subset V$
(or equivalently for every net $\{x_{i}\}_{i\in I}$ with $x_{i}\rightarrow x$,
strongly in $X$, eventually $Tx_{i}\subset V$). The operator $T$
is \emph{$s\times w^{*}-$usc} if $T$ is $s\times w^{*}-$usc at
$x$, for every $x\in X$ . 
\end{description}
It is easily checked that every representable operator is monotone
demiclosed and has $w^{*}-$closed convex values.

Recall that $T\in\mathscr{M}(X)$ iff $T$ is NI and representable
(see \cite{MR2207807} or \cite{MR2453098}); also, whenever $T\in\mathscr{M}(X)$ 
\begin{description}
\item [{$T$}] is dual-representable with $\varphi_{T},\psi_{T}\in\mathcal{D}_{T}$
(see \cite[Theorems 2.1, 2.2]{MR2207807}),
\item [{$\operatorname*{ri}D(T)={}^{{\rm ic}}D(T)={}^{{\rm ic}}(\operatorname*{ri}D(T))$,}] in
particular $\operatorname*{int}D(T)=\operatorname*{core}D(T)$ and
$\overline{D(T)}$ is convex whenever $^{{\rm ic}}(\operatorname*{conv}D(T))\neq\emptyset$
(see \cite[Corollary\ 3]{MR2577332}). 
\end{description}
Some of the main characterizations of maximal monotonicity can be
found in \cite[Theorem 3.8]{MR1009594}\emph{, }\cite[Theorem 2.3]{MR2207807},
\cite[Theorem\ 6]{MR2594359}, \cite[Lemma 4.2]{MR2453098}, \cite[Theorem 40.2, p. 155]{MR1723737},
\cite[Lemma 7.7, p. 104]{MR984602} and they will often be recalled
in the sequel. For other properties of the notions discussed in this
section we suggest \cite{MR2594359,MR2577332,MR2389004,MR2453098,MR2207807,astfnc}.

Since the characterization of maximality for a monotone operator with
a singleton domain is trivial, in this paper we do not consider singleton-domain
operators.

Throughout this article the conventions $\sup\emptyset=-\infty$ and
$\inf\emptyset=\infty$ are enforced.

\section{Restrictions to affine sets}

For $X$ a separated locally convex space and $F\subset X$ a linear
subspace with $D(T)\cap F\neq\emptyset$, let $T_{F}:=\iota_{F}^{*}T\iota_{F}:F\rightrightarrows F^{*}$
\begin{equation}
T_{F}x:=\{x^{*}|_{F}\mid x^{*}\in Tx\},\ x\in D(T_{F}):=D(T)\cap F.\label{def TF}\end{equation}
 Here $\iota_{F}^{*}:X^{*}\rightarrow F^{*}$, $\iota_{F}^{*}(x^{*})=x^{*}|_{F}$
stands for the adjoint of $\iota_{F}:F\rightarrow X$, $\iota_{F}(x)=x$.

\begin{lemma} \label{sub}Let $X$ be a normed space, let $T:X\rightrightarrows X^{*}$,
and let $F\subset X$ be a linear subspace such that $D(T)\cap F\neq\emptyset$.

Consider the conditions\emph{: (i)} $T_{F}\in\mathscr{M}(F)$; \emph{(ii)}
$T+N_{F}\in\mathscr{M}(X)$.

Then \emph{(ii) $\Rightarrow$ (i). If, in addition, $F$ is closed
then (i) $\Rightarrow$ (ii). }$\ $\end{lemma}

\begin{proof} It is easily checked that $T_{F}\in\mathcal{M}(F)$
iff $T|_{F}\in\mathcal{M}(X)$ iff $T+N_{F}\in\mathcal{M}(X)$. Here
$T|_{F}$ stands for the restriction of $T$ to $F$.

(ii) $\Rightarrow$ (i) Let $(f,f^{*})\in F\times F^{*}$ be m.r.
to $T_{F}$, i.e., \begin{equation}
\langle f-x,f^{*}-x^{*}|_{F}\rangle\ge0,\ \forall x\in D(T)\cap F,\ x^{*}\in Tx.\label{ssmrt}\end{equation}
 Take $y^{*}\in X^{*}$ with $y^{*}|_{F}=f^{*}$. Relation (\ref{ssmrt})
implies that $(f,y^{*})$ is m.r. to $T+N_{F}\in\mathscr{M}(X)$;
from which $f\in D(T)$ and $y^{*}\in Tf+F^{\perp}$. This yields
$(f,f^{*})\in T_{F}$. 

(i) $\Rightarrow$ (ii) Assume that $F$ is closed and $T_{F}\in\mathscr{M}(F)$.
Let $(x_{0},x_{0}^{*})$ be m.r. to $T+N_{F}$, that is,\begin{equation}
\langle x_{0}-x,x_{0}^{*}-x^{*}-f^{*}\rangle\ge0,\ \forall x\in D(T)\cap F,\ x^{*}\in Tx,\ f^{*}\in N_{F}(x)=F^{\perp}.\label{smrt}\end{equation}
 Since $F^{\perp}$ is a linear subspace, $x_{0}-x\in F^{\perp\perp}=F$
for every $x\in D(T)\cap F$; in particular $x_{0}\in F$. Relation
(\ref{smrt}) becomes $\langle x_{0}-x,x_{0}^{*}-x^{*}\rangle\ge0$,
for every $x\in D(T)\cap F$, $x^{*}\in Tx$. In other words $(x_{0},x_{0}^{*}|_{F})$
is m.r. to $T_{F}$ and so $x_{0}\in D(T)$ and there is $x_{1}^{*}\in Tx_{0}$
such that $x_{0}^{*}|_{F}=x_{1}^{*}|_{F}$. This last equality is
equivalent to $x_{0}^{*}-x_{1}^{*}\in F^{\perp}=N_{F}(x_{0})$. \end{proof}

\begin{remark} Notice that for $T=X\times\{0\}$, $T_{F}=F\times\{0\}\in\mathscr{M}(F)$,
for every linear subspace $F\subset X$, while $T+N_{F}=F\times F^{\perp}\in\mathscr{M}(X)$
iff $F$ is closed. Therefore a closed $F$ is necessary and sufficient
for (i) $\Rightarrow$ (ii) to hold. \end{remark} 

\begin{lemma} \label{sat} Let $X$ be a  normed space, let $T:X\rightrightarrows X^{*}$,
and let $F$ be a linear subspace such that $D(T)\subset F$.

Consider the conditions\emph{: (i)} $T\in\mathscr{M}(X)$; \emph{(ii)}
$T_{F}\in\mathscr{M}(F)$ and $T=T+N_{F}$. 

Then \emph{(i) $\Rightarrow$ (ii). If, in addition, $F$ is closed
then (ii) $\Rightarrow$ (i).}\end{lemma}

\begin{proof} (i) $\Rightarrow$ (ii) Since $\mathscr{M}(X)\ni T\subset T+N_{F}\in\mathcal{M}(X)$
we find $T=T+N_{F}\in\mathscr{M}(X)$. Hence, from Lemma \ref{sub},
$T_{F}\in\mathscr{M}(F)$. 

\noindent (ii) $\Rightarrow$ (i) According to Lemma \ref{sub}, $T=T+N_{F}\in\mathscr{M}(X)$
since $T_{F}\in\mathscr{M}(F)$. \end{proof}

\begin{remark} For a linear subspace $F\subset X$, take $T=N_{F}=F\times F^{\perp}$.
Then $T_{F}=F\times\{0\}\in\mathscr{M}(F)$, $T=T+N_{F}$, but $T\in\mathscr{M}(X)$
iff $F$ is closed. Therefore, in the previous lemma, the implication
(ii) $\Rightarrow$ (i) requires $F$ to be closed.\end{remark} 

For $T:X\rightrightarrows X^{\ast}$ and $z\in X$, consider the translation
$T_{z}(x):=T(x+z)$, $x\in D(T_{z})=D(T)-z$. Similarly, for an affine
subset $A\subset X$ and $z\in A$ let $F:=A-z$ be the linear subspace
parallel to $A$. Note that $\operatorname*{Graph}N_{A}=A\times F^{\perp}$.
By definition $T_{A,z}:F\rightrightarrows F^{*}$, $T_{A,z}:=(T_{z})_{F}$
or equivalently \begin{equation}
\operatorname*{Graph}T_{A,z}=\{(x,x^{*}|_{F})\mid x^{*}\in T(z+x),\ x\in D(T_{A,z}):=(D(T)-z)\cap F\}.\label{def TAz}\end{equation}
 Notice that $T_{A,z}$ is non-empty iff $D(T)\cap A\neq\emptyset$.
Also, note that $T_{F,z}=(T_{F})_{z}$, for every $z\in F$.

\strut

The following two results are consequences of the previous lemmas
and the fact that a translation preserves the NI type and the (maximal)
monotonicity; in other words $T$ is NI iff $T_{z}$ is NI and $T$
is (maximal) monotone iff $T_{z}$ is (maximal) monotone, for every
(some) $z\in X$.

\begin{lemma} \label{afin}Let $X$ be a  normed space, let $T:X\rightrightarrows X^{*}$,
let $A\subset X$ be an affine set such that $D(T)\cap A\neq\emptyset$,
and let $F$ be the linear subspace parallel to $A$. If $T+N_{A}\in\mathscr{M}(X)$
then $T_{A,z}\in\mathscr{M}(F)$, for every $z\in A$. If, in addition,
$A$ is closed then $T+N_{A}\in\mathscr{M}(X)$ whenever $T_{A,z}\in\mathscr{M}(F)$,
for some $z\in A$. 

\end{lemma}

\begin{proof} If $T+N_{A}\in\mathscr{M}(X)$ then $(T+N_{A})_{z}=T_{z}+N_{F}\in\mathscr{M}(X)$.
According to Lemma \ref{sub}, $T_{A,z}\in\mathscr{M}(F)$ for every
$z\in A$.

Conversely, if $A$ is closed and $T_{A,z}=(T_{z})_{F}\in\mathscr{M}(F)$
for some $z\in A$ then, from Lemma \ref{sub}, $T_{z}+N_{F}=(T+N_{A})_{z},T+N_{A}\in\mathscr{M}(X)$.
\end{proof}

\strut

Similarly, one has

\begin{lemma} \label{afinr} Let $X$ be a  normed space, let $T:X\rightrightarrows X^{*}$,
let $A\subset X$ be an affine set such that $D(T)\subset A$, and
let $F$ be the linear subspace parallel to $A$. If $T\in\mathscr{M}(X)$
then $T=T+N_{A}$ and $T_{A,z}\in\mathscr{M}(F)$, for every $z\in A$.
If, in addition, $A$ is closed then $T\in\mathscr{M}(X)$ whenever
$T=T+N_{A}$ and $T_{A,z}\in\mathscr{M}(F)$, for some $z\in A$.
\end{lemma}

Hereditary properties from $T$ to $T_{A,z}$, where $A$ is affine
with $D(T)\subset A$, are studied next.

\begin{proposition} \label{val-inchise} Let $X$ be a separated
locally convex space, let $T:X\rightrightarrows X^{*}$ be such that
$T=T+N_{A}$, where $A\subset X$ is affine with $D(T)\subset A$,
and let $F$ be the linear subspace parallel to $A$. Then $T$ has
$w^{*}-$closed values in $X^{*}$ whenever $T_{A,z}$ has $w^{*}-$closed
values in $F^{*}$, for some $z\in A$. If, in addition, $A$ is closed
then $T_{A,z}$ has $w^{*}-$closed values in $F^{*}$, for every
$z\in A$ whenever $T$ has $w^{*}-$closed values in $X^{*}$. \end{proposition}

\begin{proof} We may assume without loss of generality that $z=0\in A=F$
in which case $T_{A,z}=T_{F}$. 

It suffices to observe that $(\iota_{F}^{*})^{-1}(T_{F}x)=Tx+F^{\perp}=Tx$,
for every $x\in D(T)$ to conclude that $T$ has $w^{*}-$closed values
in $X^{*}$ provided that $T_{F}$ has $w^{*}-$closed values in $F^{*}$.

Conversely, assume that $F$ is closed. Let $\pi:(X^{*},w^{*})\rightarrow(X^{*},w^{*})/F^{\perp}$,
$\pi(x^{*})=x^{*}+F^{\perp}$, $x^{*}\in X^{*}$, be the projection
map of $(X^{*},w^{*})$ onto the quotient space $(X^{*},w^{*})/F^{\perp}$.
Note that $\pi^{-1}(\pi(Tx))=Tx+F^{\perp}=Tx$ is $w^{*}-$closed
convex, i.e., $\pi(Tx)$ is closed convex in $(X^{*},w^{*})/F^{\perp}$,
for every $x\in D(T)$.

Since $F$ is closed in $X$ the locally convex spaces $(X^{*},w^{*})/F^{\perp}$
and $(F^{*},w^{*})$ are isomorphic via $j:(X^{*},w^{*})/F^{\perp}\rightarrow(F^{*},w^{*})$,
$j(x^{*}+F^{\perp})=x^{*}|_{F}$, $x^{*}\in X^{*}$. Notice that $T_{F}x=\iota_{F}^{*}(Tx)=j(\pi(Tx))$,
$x\in D(T)$, to conclude that $T_{F}$ has $w^{*}-$closed convex
values. \end{proof}

\begin{remark} If $T$ has $w^{*}-$closed values in $X^{*}$ but
$F$ is not closed then one cannot expect $T_{F}$ to have $w^{*}-$closed
values in $F^{*}$ even though $T=T+N_{F}$. Indeed, let $F$ be a
proper dense linear subspace of a separated locally convex space $X$.
Then $\iota_{F}^{*}:(X^{*},w^{*})\rightarrow(F^{*},w^{*})$ is a continuous
bijection but not an isomorphism, since $F\subsetneq X$. Take $S$
a $w^{*}-$closed subset of $X^{*}$ such that $\iota_{F}^{*}(S)$
is not $w^{*}-$closed in $F^{*}$ and $T:=\{0\}\times S\subset X\times X^{*}$.
Then $T_{F}=\{0\}\times\iota_{F}^{*}(S)$ does not have $w^{*}-$closed
values (and implicitly is not $s\times w^{*}-$closed in $F\times F^{*}$)
while $T$ is $s\times w^{*}-$closed in $X\times X^{*}$, has $w^{*}-$closed
values, and $T=T+N_{F}$, since $N_{F}=F\times\{0\}$. \end{remark} 

\begin{remark} Let $T:X\rightrightarrows X^{*}$ and let $A\subset X$
be affine with $D(T)\subset A$. It is trivial to check that $T_{A,z}$
has convex values, for every $z\in A$ whenever $T$ has convex values.
Conversely, $T$ has convex values whenever $T_{A,z}$ has convex
values, for some $z\in A$, provided, in addition, that $T=T+N_{A}$.
\end{remark} 

\begin{proposition} \label{sxw*}Let $X$ be a separated locally
convex space, let $T:X\rightrightarrows X^{*}$ be such that $T=T+N_{A}$,
where $A\subset X$ is an affine set such that $D(T)\subset A$, and
let $F$ be the linear subspace parallel to $A$. Then $T$ is $s\times w^{*}-$closed
in $X\times X^{*}$ whenever $T_{A,z}$ is $s\times w^{*}-$closed
in $F\times F^{*}$, for some $z\in A$. If, in addition, $A$ is
closed then $T_{A,z}$ is $s\times w^{*}-$closed in $F\times F^{*}$,
for every $z\in A$ provided that $T$ is $s\times w^{*}-$closed
in $X\times X^{*}$. \end{proposition}

\begin{proof} We may assume without loss of generality that $z=0\in A=F$.
In this case $T_{A,z}=T_{F}$. Let $L:(X\times X^{*},s\times w^{*})\rightarrow(X\times F^{*},s\times w^{*})$,
$L(x,x^{*})=(x,x^{*}|_{F})$. Then $L$ is linear bounded and $\operatorname*{Graph}T=L^{-1}(\operatorname*{Graph}T_{F})$
since $T=T+N_{F}$. Therefore $T$ is $s\times w^{*}-$closed in $X\times X^{*}$
whenever $T_{F}$ is $s\times w^{*}-$closed in $F\times F^{*}$.

Assume that $F$ is closed and $T$ is $s\times w^{*}-$closed in
$X\times X^{*}$. Let $\Pi:(X\times X^{*},s\times w^{*})\rightarrow(X\times X^{*},s\times w^{*})/\{0\}\times F^{\perp}$,
$\Pi(x,x^{*})=(x,x^{*})+\{0\}\times F^{\perp}$, be the projection
map of $(X\times X^{*},s\times w^{*})$ onto the quotient space $(X\times X^{*},s\times w^{*})/\{0\}\times F^{\perp}$.
Note that $\Pi^{-1}(\Pi(\operatorname*{Graph}T))=\operatorname*{Graph}T+\{0\}\times F^{\perp}=\operatorname*{Graph}T$
is $s\times w^{*}-$closed, that is, $\Pi(\operatorname*{Graph}T)$
is closed in $(X\times X^{*},s\times w^{*})/\{0\}\times F^{\perp}$.
The locally convex spaces $(X\times X^{*},s\times w^{*})/\{0\}\times F^{\perp}$
and $(X\times F^{*},s\times w^{*})$ are isomorphic via $j((x,x^{*})+\{0\}\times F^{\perp})=(x,x^{*}|_{F})$
since $F$ is closed. Using $L=j\circ\Pi$ we get that $\operatorname*{Graph}T_{F}=L(\operatorname*{Graph}T)=j(\Pi(\operatorname*{Graph}T))$
is $s\times w^{*}-$closed in $F\times F^{*}$. \end{proof}

\begin{proposition} \label{demi}Let $(X,\|\cdot\|)$ be a normed
space, let $T:X\rightrightarrows X^{*}$, let $A\subset X$ be an
affine set such that $D(T)\subset A$, and let $F$ be the linear
subspace parallel to $A$. If $T$ is demiclosed in $X\times X^{*}$
then $T_{A,z}$ is demiclosed in $F\times F^{*}$, for every $z\in A$.
If, in addition, $T=T+N_{A}$, then $T$ is demiclosed in $X\times X^{*}$
whenever $T_{A,z}$ is demiclosed in $F\times F^{*}$, for some $z\in A$.
\end{proposition}

\begin{proof} Up to a translation we may assume without loss of generality
that $z=0\in A=F$. In this case $T_{A,z}=T_{F}$. 

Assume that $T$ is demiclosed in $X\times X^{*}$. Let $\{(x_{i},f_{i}^{*})\}_{i\in I}\subset T_{F}$
be such that $(x_{i},f_{i}^{*})\rightarrow(x,f^{*})$, $s\times w^{*}$
in $F\times F^{*}$ and $(f_{i}^{*})_{i}$ is (strongly) bounded in
$F^{*}$. According to the Hahn-Banach Theorem, for every $i\in I$
there is $x_{i}^{*}\in X^{*}$ such that $x_{i}^{*}|_{F}=f_{i}^{*}$
and $\|x_{i}^{*}\|=\|f_{i}^{*}\|$. Hence $(x_{i}^{*})_{i}$ is bounded
in $X^{*}$ and, eventually on a subnet, denoted by the same index
for simplicity, $(x_{i},x_{i}^{*})\rightarrow(x,x^{*})$, $s\times w^{*}$
in $X\times X^{*}$. This yields that $x^{*}|_{F}=f^{*}$ and $x^{*}\in Tx$
because $T$ is demiclosed; whence $f^{*}\in T_{F}x$. Therefore $T_{F}$
is demiclosed in $F\times F^{*}$.

Assume that $T=T+N_{F}$ and $T_{F}$ is demiclosed in $F\times F^{*}$.
Let $\{(x_{i},x_{i}^{*})\}_{i\in I}\subset T$ be such that $(x_{i},x_{i}^{*})\rightarrow(x,x^{*})$,
$s\times w^{*}$ in $X\times X^{*}$ and $(x_{i}^{*})_{i}$ is bounded
in $X^{*}$. Then $(x_{i},x_{i}^{*}|_{F})\rightarrow(x,x^{*}|_{F})$,
$s\times w^{*}$ in $F\times F^{*}$ and $(x_{i}^{*}|_{F})_{i}$ is
bounded in $F^{*}$. Due to the demiclosedness of $T_{F}$ we find
that $x^{*}|_{F}\in T_{F}x$, that is, there is $y^{*}\in Tx$ such
that $x^{*}|_{F}=y^{*}|_{F}$. This implies $x^{*}\in y^{*}+F^{\perp}\subset Tx$.
\end{proof}

\begin{remark} Under all the assumptions in Proposition \ref{demi},
$T_{F}$ is demiclosed in $F^{*}$ on a subset $S\subset\overline{D(T)}$
iff $T$ is demiclosed in $X^{*}$ on $S$ (that is, if $\{(x_{i},x_{i}^{*})\}_{i}\subset T$,
$(x_{i},x_{i}^{*})\rightarrow(x,x^{*})$ $s\times w^{*}$ in $X\times X^{*}$
and $x\in S$ then $(x,x^{*})\in T$).\end{remark} 

\begin{corollary} \label{swr}Let $X$ be a normed space, let $T:X\rightrightarrows X^{*}$
be such that $\operatorname*{aff}D(T)$ is closed and $T=T+N_{\operatorname*{aff}D(T)}$,
and let $F$ be the linear subspace parallel to $\operatorname*{aff}D(T)$.
Then

\emph{(i)} $T$ has $w^{*}-$closed values in $X^{*}$ iff $T_{\operatorname*{aff}D(T),z}$
has $w^{*}-$closed values in $F^{*}$, for every (some) $z\in\operatorname*{aff}D(T)$;

\emph{(ii)} $T$ is $s\times w^{*}-$closed in $X\times X^{*}$ iff
$T_{\operatorname*{aff}D(T),z}$ is $s\times w^{*}-$closed in $F\times F^{*}$,
for every (some) $z\in\operatorname*{aff}D(T)$; 

\emph{(iii)} $T$ is demiclosed in $X\times X^{*}$ iff $T_{\operatorname*{aff}D(T),z}$
is demiclosed in $F\times F^{*}$, for every (some) $z\in\operatorname*{aff}D(T)$.
\end{corollary}

In the case of a finite-dimensional affine set passing through $z$
and being spanned by the linearly independent set of directions $\{v_{1},v_{2},\ldots,v_{d}\}$
\begin{equation}
A:=A(z;v_{1},v_{2},..,v_{d}):=\{x=z+t_{1}v_{1}+..+t_{d}v_{d}\mid(t_{1},..,t_{d})\in\mathbb{R}^{d}\},\label{def A fin dim}\end{equation}
we associate to $T_{A,z}$ the finite-dimensional operator $\mathscr{T}_{A,z}:\mathbb{R}^{d}\rightrightarrows\mathbb{R}^{d}$
given by \begin{equation}
(s_{1},..,s_{d})\in\mathscr{T}_{A,z}(t_{1},..,t_{d})\ {\rm iff}\ \exists x^{*}\in T(z+t_{1}v_{1}+..+t_{d}v_{d}):\,\, s_{i}=\langle v_{i},x^{*}\rangle,\,\, i=\overline{1,d}.\label{def TAz fin dim}\end{equation}
 Note that $\mathscr{T}_{A,z}=\mathscr{I}^{*}T_{A,z}\mathscr{I}$,
where $\mathscr{I}:\mathbb{R}^{d}\rightarrow F:=\operatorname*{span}\{v_{1},..,v_{d}\}$
is the isomorphism given by $\mathscr{I}(t_{1},..,t_{d}):=t_{1}v_{1}+..+t_{d}v_{d}$.
This latter operator identity provides \begin{equation}
\varphi_{\mathscr{T}_{A,z}}((t_{1},..,t_{d}),(s_{1},..,s_{d}))=\varphi_{T_{A,z}}(t_{1}v_{1}+..+t_{d}v_{d},x^{*}),\label{fi-rel}\end{equation}
where $x^{*}\in F^{*}$ is uniquely determined by $s_{i}=\langle v_{i},x^{*}\rangle$,
$i=\overline{1,d}$.

Alternately, $\mathscr{T}_{A,z}=J^{*}T_{z}J:D(\mathscr{T}_{A,z}):=J^{-1}(D(T)-z)\subset\mathbb{R}^{d}\rightrightarrows\mathbb{R}^{d}$,
where $J:=\iota_{F}\circ\mathscr{J}:\mathbb{R}^{d}\rightarrow X$,
$J(t_{1},..,t_{d}):=t_{1}v_{1}+..+t_{d}v_{d}$.

Notice also that $\mathscr{T}_{A,z_{2}}=(\mathscr{T}_{A,z_{1}})_{J^{-1}(z_{2}-z_{1})}$,
for every $z_{1},z_{2}\in A$. 

\begin{lemma} \label{red-fin dim} Let $X$ be a  normed space, $T:X\rightrightarrows X^{*}$,
$z\in X$, $\{v_{1},..,v_{d}\}$ be a linearly independent subset
of $X$, $F=\operatorname*{span}\{v_{1},..,v_{d}\}$, and $A=A(z;v_{1},..,v_{d})$.
Then

\emph{(i)} $\mathscr{T}_{A,z}$ is NI iff $T_{A,z}$ is NI,

\medskip

\emph{(ii)} $\mathscr{T}_{A,z}\in\mathscr{M}(\mathbb{R}^{d})$ iff
$T_{A,z}\in\mathscr{M}(F)$ iff $T+N_{A}\in\mathscr{M}(X)$.

\end{lemma}

In the sequel, for $v\neq0$, $z\in X$, $L:=L(z;v):=A(z;v)$ is the
line passing through $z$ with direction $v$, $T_{L,z}(tv):=\{x^{*}|_{\mathbb{R}v}\mid x^{*}\in T(z+tv)\}$,
$t\in\mathbb{R}$, and $\mathscr{T}_{L,z}:\mathbb{R}\rightrightarrows\mathbb{R}$,
$s\in\mathscr{T}_{L,z}t$ if there is $x^{*}\in T(z+tv)$ such that
$\langle v,x^{*}\rangle=s$.

\medskip

Similarly, the plane passing through $z$ with linearly independent
set of directions $\{v_{1},v_{2}\}$ is given by $P:=P(z;v_{1},v_{2}):=A(z;v_{1},v_{2})$,
$T_{P,z}(t_{1}v_{1}+t_{2}v_{2})=\{x^{*}|_{\operatorname*{span}\{v_{1},v_{2}\}}\mid x^{*}\in T(z+t_{1}v_{1}+t_{2}v_{2})\}$,
$t_{1},t_{2}\in\mathbb{R}$, and $\mathscr{T}_{P,z}:\mathbb{R}^{2}\rightrightarrows\mathbb{R}^{2}$,
$(s_{1},s_{2})\in\mathscr{T}_{P,z}(t_{1},t_{2})$ if there is $x^{*}\in T(z+t_{1}v_{1}+t_{2}v_{2})$
such that $\langle v_{1},x^{*}\rangle=s_{1}$, $\langle v_{2},x^{*}\rangle=s_{2}$.

\begin{proposition}\label{cl-val-graph} Let $(X,\|\cdot\|)$ be
a Banach space, let $T\in\mathcal{M}(X)$ be such that $\operatorname*{ri}D(T)\neq\emptyset$
and let $A\subset\operatorname*{aff}D(T)$ be finite-dimensional affine
with $A\cap\operatorname*{ri}D(T)\neq\emptyset$. 

\emph{(i)} If $T$ has $w^{*}-$closed values then ${\mathscr{T}}_{A,z}$
has closed values, for every $z\in A$.

\medskip

\emph{(ii)} If $T$ is demiclosed then ${\mathscr{T}}_{A,z}$ is closed,
for every $z\in A$. \end{proposition}

\begin{proof} We may assume without loss of generality that $\operatorname*{aff}D(T)=X$,
otherwise we replace $T$ by $T_{\operatorname*{aff}D(T),z}$ and
acknowledge Corollary \ref{swr} . This change does not affect ${\mathscr{T}}_{A,z}$.

Let $A=A(z;v_{1},..,v_{d})$ and $z_{0}:=z+t_{1}^{0}v_{1}+..+t_{d}^{0}v_{d}\in A\cap\operatorname*{int}D(T)$.
From the local boundedness of $T$ at $z_{0}$ there are $M,r>0$
such that $z_{0}+rB_{X}\subset D(T)$ and $\|x^{*}\|\le M$, for every
$x^{*}\in T(z_{0}+ru)$, $u\in B_{X}$. 

(ii) Consider $\{(s_{1}^{n},..,s_{d}^{n}),(t_{1}^{n},..,t_{d}^{n})\}_{n\ge1}\subset\mathscr{T}_{A,z}$,
that is, there exists $x_{n}^{*}\in T(z+t_{1}^{n}v_{1}+..+t_{d}^{n}v_{d})$
such that $\langle v_{i},x_{n}^{*}\rangle=s_{i}^{n}$, $i=\overline{1,d}$,
$n\ge1$. Assume that $\lim_{n\rightarrow\infty}((s_{1}^{n},..,s_{d}^{n}),(t_{1}^{n},..,t_{d}^{n}))=((s_{1},..,s_{d}),(t_{1},..,t_{d}))$.

The monotonicity of $T$ provides\[
\langle(t_{1}^{n}-t_{1}^{0})v_{1}+..+(t_{d}^{n}-t_{d}^{0})v_{d}-ru,x_{n}^{*}-x^{*}\rangle\ge0,\ \forall x^{*}\in T(z_{0}+ru),u\in B_{X},n\ge1,\]
from which $r\|x_{n}^{*}\|\le M(|t_{1}^{n}-t_{1}^{0}|\|v_{1}\|+..+|t_{d}^{n}-t_{d}^{0}|\|v_{d}\|+r)+(t_{1}^{n}-t_{1}^{0})s_{1}^{n}+..+(t_{d}^{n}-t_{d}^{0})s_{d}^{n}$,
$n\ge1$. Hence $\{x_{n}^{*}\}_{n\ge1}$ is bounded. On a subnet denoted
for simplicity by the same index, $x_{n}^{*}\rightarrow x^{*}\in T(z+t_{1}v_{1}+..+t_{d}v_{d})$
weakly-star in $X^{*}$, since $T$ is demiclosed. Let $n\rightarrow\infty$
in $\langle v_{i},x_{n}^{*}\rangle=s_{i}^{n}$, $i=\overline{1,d}$,
to get $\langle v_{i},x^{*}\rangle=s_{i}$, $i=\overline{1,d}$, that
is $(s_{1},..,s_{d})\in\mathscr{T}_{A,z}(t_{1},..,t_{d})$. 

The argument of (i) proceeds similarly with $(t_{1}^{n},..,t_{d}^{n})=(t_{1},..,t_{d})\in D({\mathscr{T}}_{A,z})$,
$n\ge1$. \end{proof}

\strut

Since every convex set in a finite dimensional space has a non-empty
relative (algebraic) interior which is dense in the set, the following
definition is a natural extension to a general topological vector
space context for the nearly-convex notion. A set $S\subset X$ is
called \emph{nearly-convex} if there is a convex set $C$ such that
$\operatorname*{ri}C\neq\emptyset$ and $C\subset S\subset\overline{C}$.
Equivalently, $S$ is nearly-convex iff $\operatorname*{ri}S$ is
non-empty convex and $S\subset\operatorname*{cl}(\operatorname*{ri}S)$.
Indeed, directly, we know that $\operatorname*{ri}C=\operatorname*{ri}\overline{C}$
and $\operatorname*{cl}(\operatorname*{ri}C)=\overline{C}$ (see \cite[Lemma 11A b), p. 59]{MR0410335})
from which $\operatorname*{ri}S=\operatorname*{ri}C$ is non-empty
convex and $S\subset\overline{C}=\operatorname*{cl}(\operatorname*{ri}S)$.
Conversely, $C=\operatorname*{ri}S$ fulfills all the required conditions.

\begin{lemma} \label{int-cl-fin-dim}Let $X$ be a Banach space,
let $T:X\rightrightarrows X^{*}$ be such that $D(T)$ is nearly-convex,
and let $A=A(z;v_{1},v_{2},..,v_{d})\subset\operatorname*{aff}D(T)$
be such that $A\cap\operatorname*{ri}D(T)\neq\emptyset$. Then $D(\mathscr{T}_{A,z})$
is nearly-convex and \begin{equation}
\operatorname*{int}D(\mathscr{T}_{A,z})=J^{-1}(\operatorname*{ri}D(T)-z),\ \operatorname*{cl}D(\mathscr{T}_{A,z})=J^{-1}(\overline{D(T)}-z),\label{icfd}\end{equation}
where $J(t_{1},..,t_{d})=t_{1}v_{1}+..+t_{d}v_{d}$, $(t_{1},..,t_{d})\in\mathbb{R}^{d}$.
\end{lemma}

\begin{proof} We may assume without loss of generality that $\operatorname*{aff}D(T)=X$
(otherwise we replace $X$ by $F:=\operatorname*{aff}D(T)\ni0$ and
$T$ by $T_{F}$; $\mathscr{T}_{A,z}$ being impervious to this change).
Let $D:=\operatorname*{int}D(T)-z$ and $S:=D(T)-z$. Then $J^{-1}(D)$
is non-empty open convex, $J^{-1}(S)=D(\mathscr{T}_{A,z})$, $\overline{D}=\overline{S}$,
and $\operatorname*{int}\overline{S}=D$ (see e.g. \cite[Lemma 11A b), p. 59]{MR0410335}).

Fix $\bar{t}$ with $J\bar{t}\in D$. Since $J^{-1}(D)\subset D(\mathscr{T}_{A,z})\subset J^{-1}(\overline{D})$
and $J^{-1}(D)\subset\operatorname*{int}J^{-1}(\overline{D})$, $\operatorname*{cl}J^{-1}(D)\subset J^{-1}(\overline{D})$
(due to the continuity of $J$) in order to conclude it suffices to
show that $\operatorname*{int}J^{-1}(\overline{D})\subset J^{-1}(D)$
and $J^{-1}(\overline{D})\subset\operatorname*{cl}J^{-1}(D)$. 

Let $t\in J^{-1}(\overline{D})$. Since $\operatorname*{int}\overline{D}=D$,
$(1-\lambda)\bar{t}+\lambda t\in J^{-1}(D)$, for $0\le\lambda<1$.
Let $\lambda\uparrow1$ to get $t\in\operatorname*{cl}J^{-1}(D)$.
Hence $J^{-1}(\overline{D})=\operatorname*{cl}J^{-1}(D)$ followed
by $\operatorname*{int}J^{-1}(\overline{D})=\operatorname*{int}(\operatorname*{cl}J^{-1}(D))=J^{-1}(D)$.
\end{proof}

\begin{lemma}\label{her-cone} Let $X$ be a Banach space and let
$T:X\rightrightarrows X^{*}$ be such that $D(T)$ is nearly-convex
and $T=T+N_{D(T)}$. Then $\mathscr{T}_{A,z}=\mathscr{T}_{A,z}+N_{D(\mathscr{T}_{A,z})}$,
for every finite-dimensional affine set $A\subset X$ with $A\cap\operatorname*{ri}D(T)\neq\emptyset$,
$z\in A\cap\operatorname*{ri}D(T)$. \end{lemma}

\begin{proof} Condition $T=T+N_{D(T)}$ provides $T_{z}=T_{z}+(N_{D(T)})_{z}=T_{z}+N_{D(T)-z}$,
$z\in X$. Recall that $\mathscr{T}_{A,z}=J^{*}T_{z}J$, where $A=A(z;v_{1},..,v_{d})$,
$J:\mathbb{R}^{d}\rightarrow X$, $J\widehat{t}=t_{1}v_{1}+..+t_{d}v_{d}$,
$\widehat{t}=(t_{1},..,t_{d})$. Hence \[
\mathscr{T}_{A,z}\widehat{t}=\mathscr{T}_{A,z}\widehat{t}+J^{*}N_{D(T)-z}J\widehat{t},\ \forall\widehat{t}\in D(\mathscr{T}_{A,z})=J^{-1}(D(T)-z).\]

But, for every $\widehat{t}\in D(\mathscr{T}_{A,z})$ we have from
(\ref{icfd}) and by the chain rule (see e.g. \cite[Theorem 2.8.3(vii), p. 123]{MR1921556}),
that \[
N_{D({\mathscr{T}}_{A,z})}\widehat{t}=N_{\operatorname*{cl}D({\mathscr{T}}_{A,z})}\widehat{t}=N_{J^{-1}(\overline{D(T)}-z)}\widehat{t}=\partial(\iota_{\overline{D(T)}-z}\circ J)(\widehat{t})=J^{*}N_{\overline{D(T)}-z}J\widehat{t},\]
since $0\in{}^{ic}(A-\overline{D(T)})$. The proof is complete. \end{proof}

\section{Finite-dimensional context characterizations}

Recall the next sum rule for maximal monotone operators followed by
two of its consequences.

\begin{theorem}\emph{(\cite[Corollary\ 4]{MR2577332})\label{mmcs}}
Let $X$ be a Banach space and $M,N\in\mathscr{M}(X)$. Assume that
$^{ic}D(M),{}^{ic}D(N)$ are nonempty and $0\in{}^{ic}(D(M)-D(N))$.
Then $M+N\in\mathscr{M}(X)$.

\end{theorem}

\begin{proposition} \label{afin-fin dim}Let $X$ be a Banach space.
If $T\in\mathscr{M}(X)$, $A$ is a finite-dimensional affine subset
of $X$, and $A\cap{}^{ic}D(T)\neq\emptyset$ then $T+N_{A}\in\mathscr{M}(X)$.
\end{proposition}

\begin{proof} We may apply the previous theorem since $^{ic}A=A$,
$\operatorname*{aff}(D(T)-A)=\operatorname*{aff}D(T)-A$ is closed
since $\operatorname*{aff}D(T)$ is closed and $A$ is finite-dimensional,
and $0\in{}^{i}D(T)-A\subset{}^{i}(D(T)-A)$.

Indeed, for the latter inclusion let $x_{0}\in{}^{ic}D(T)$, $a_{0}\in A$.
Every $x\in\operatorname*{aff}(D(T)-A)$ has the form $x=u-a$, with
$u\in\operatorname*{aff}D(T)$, $a\in A$. Therefore there is $\rho>0$
such that $\lambda u+(1-\lambda)x_{0}\in D(T)$, for every $\lambda\in[0,\rho]$.
Hence $\lambda x+(1-\lambda)(x_{0}-a_{0})=\lambda u+(1-\lambda)x_{0}-(\lambda a+(1-\lambda)a_{0})\in D(T)-A$,
for every $\lambda\in[0,\rho]$, that is, $x_{0}-a_{0}\in{}^{ic}(D(T)-A)$.
The proof is complete. \end{proof}

\begin{proposition} \label{afin-int}Let $X$ be a Banach space.
If $T\in\mathscr{M}(X)$, $A$ is a closed affine subset of $X$,
and $A\cap\operatorname*{core}D(T)\neq\emptyset$ then $T+N_{A}\in\mathscr{M}(X)$.
\end{proposition}

\noindent This paper is mainly concerned with the following converse
of Proposition \ref{afin-fin dim}.

\medskip

\noindent \psframebox[shadow=true,shadowsize=4pt ]{

\begin{tabular}{l}
\textsf{\emph{Given $T\in\mathcal{M}(X)$ with the property that $T+N_{A}\in\mathscr{M}(X)$
for every finite-}}\tabularnewline
\textsf{\emph{dimensional affine $A$ (especially lines and planes)
such that $A\cap{}^{ic}D(T)\neq\emptyset$ }}\tabularnewline
\textsf{\emph{what additional conditions on $T$ are needed, such
as the closedness of its}}\tabularnewline
\textsf{\emph{graph or convexity of its values, in order to obtain
the maximality of $T$?}}\tabularnewline
\end{tabular}

\noindent \textsf{\emph{}}}

\begin{proposition} \label{1-dim-convex}Let $X$ be a Banach space
and let $T\in\mathcal{M}(X)$ be such that $^{ic}D(T)\neq\emptyset$.
If $T+N_{L}\in\mathscr{M}(X)$ for every line $L\subset\operatorname*{aff}D(T)$
such that $L\cap{}^{ic}D(T)\neq\emptyset$ then $T_{\operatorname*{aff}D(T),z}$
is NI, for every $z\in\operatorname*{aff}D(T)$ and $\operatorname*{ri}D(T)={}^{ic}D(T)$,
$\overline{D(T)}=\operatorname*{cl}({}^{ic}D(T))$ are convex sets.
In particular, $^{ic}D(T)$, $D(T)$, $\overline{D(T)}$ have the
same (convex) relative interior and closure. \end{proposition}

\begin{proof} First assume that $\operatorname*{aff}D(T)=X$. Let
$(x_{0},x_{0}^{*})$ be m.r. to $T$ and let $v\in X$ be such that
$L:=L(x_{0};v)$ cuts $^{ic}D(T)=\operatorname*{core}D(T)$. Since
$n^{*}\in N_{L}(x)$ iff $x\in L$ and $\langle v,n^{*}\rangle=0$,
it is easily checked that $(x_{0},x_{0}^{*})$ is m.r to $T+N_{L}\in\mathscr{M}(X)$.
Therefore, if $(x_{0},x_{0}^{*})$ is m.r. to $T$ then $x_{0}\in D(T)$
and for every $v$ such that $L(x_{0};v)\cap\operatorname*{core}D(T)\neq\emptyset$
there is $x^{*}\in Tx_{0}$ such that $\langle v,x_{0}^{*}\rangle=\langle v,x^{*}\rangle$.

In particular $T$ is NI (see \cite[(1)]{astfnc}) and in this case
we know that $\mathcal{R}:=[\psi_{T}=c]$ is the unique maximal monotone
extension of $T$ (see \cite[Proposition\ 4 (iii)]{MR2594359}). 

Every $(x,x^{*})\in\mathcal{R}$ is m.r. to $T$. Hence $x\in D(T)$
and for every $w$ such that $L(x;w)\cap\operatorname*{core}D(T)\neq\emptyset$
there is $y^{*}\in Tx$ such that $\langle w,x^{*}\rangle=\langle w,y^{*}\rangle$.
Therefore $D(\mathcal{R})=D(T)$.

Since $\mathcal{R}\in\mathscr{M}(X)$ and $\operatorname*{core}D(\mathcal{R})\neq\emptyset$
we have $\operatorname*{int}D(T)=\operatorname*{core}D(T)$ and $\operatorname*{cl}D(T)=\operatorname*{cl}D(\mathcal{R})=\operatorname*{cl}(\operatorname*{int}D(\mathcal{R}))=\operatorname*{cl}(\operatorname*{int}D(T))$
are convex sets (see \cite[Corollary\ 3]{MR2577332} or \cite[Theorem\ 1]{MR0253014}).

In the general case we may assume without loss of generality that
$0\in\operatorname*{aff}D(T)=:F$ otherwise we change $T$ with $T_{z}$
where $z\in\operatorname*{aff}D(T)$. We use the first part of our
proof for $T_{F}\in\mathcal{M}(F)$. To this end note first that $T_{F}+(N_{L})_{F}=(T+N_{L})_{F}$,
where $(N_{L})_{F}\in\mathscr{M}(F)$ is the normal cone to $L\subset F$.
According to Lemma \ref{sub}, $(T+N_{L})_{F}\in\mathscr{M}(F)$,
since $(T+N_{L})+N_{F}=T+N_{L}\in\mathscr{M}(X)$, for every line
$L\subset F$, such that $L\cap{}^{ic}D(T)\neq\emptyset$. Therefore
$\operatorname*{ri}D(T)={}^{ic}D(T)$, $\overline{D(T)}=\operatorname*{cl}({}^{ic}D(T))$
are convex sets and $T_{F,z}=(T_{F})_{z}=(T_{z})_{F}$ is NI, for
every $z\in F$. The sets $^{ic}D(T)$, $D(T)$, $\overline{D(T)}$
have the same relative interior due to $\operatorname*{ri}\overline{D(T)}=\operatorname*{ri}(\operatorname*{cl}(\operatorname*{ri}D(T)))=\operatorname*{ri}D(T)$
(see e.g. \cite[Lemma 11A b), p. 59]{MR0410335}). \end{proof}

\begin{theorem} \label{c-1dim-int-fin dim}Let $T\in\mathcal{M}(\mathbb{R}^{d})$
be such that $\operatorname*{core}D(T)\neq\emptyset$. Then $T\in\mathscr{M}(X)$
iff

\emph{(i)} $T+N_{L}\in\mathscr{M}(\mathbb{R}^{d})$ for every line
$L\subset\mathbb{R}^{d}$ such that $L\cap\operatorname*{core}D(T)\neq\emptyset$,

\emph{(ii)} $T$ has closed convex values,

\emph{(iii)} $Tx=Tx+N_{\overline{D(T)}}x$, for every $x\in D(T)\setminus\operatorname*{core}D(T)$.\end{theorem}

\begin{proof} The direct implication is clear since (i) follows from
Proposition \ref{afin-fin dim} and (ii), (iii) are usual properties
of maximal monotone operators.

For the converse suppose that (i), (ii), (iii) hold. As seen in the
proof of Proposition \ref{1-dim-convex}, if $(x_{0},x_{0}^{*})$
is m.r. to $T$ then $x_{0}\in D(T)$, \begin{equation}
\forall v:\ L(x_{0};v)\cap\operatorname*{core}D(T)\neq\emptyset,\ \exists x^{*}\in Tx_{0},\ \langle v,x_{0}^{*}\rangle=\langle v,x^{*}\rangle,\label{mrt1}\end{equation}
and $\operatorname*{core}D(T)=\operatorname*{int}D(T)=\operatorname*{int}\overline{D(T)}$,
$\overline{D(T)}$ are convex sets.

Assume that $x_{0}^{*}\not\in Tx_{0}$. Since $Tx_{0}$ is closed
convex, for some $y_{0}\in\mathbb{R}^{d}$ \begin{equation}
\langle y_{0},x_{0}^{*}\rangle>\sigma_{Tx_{0}}(y_{0}):=\sup\{\langle y_{0},x^{*}\rangle\mid\ x^{*}\in Tx_{0}\},\label{sep-fin dim}\end{equation}
that is, $[\sigma_{C}<0]$ is non-empty, where $C:=Tx_{0}-x_{0}^{*}\not\ni0$.

Relation (\ref{mrt1}) shows that $L(x_{0};z)\cap\operatorname*{int}D(T)=\emptyset$
for every $z\in[\sigma_{C}<0]$. Hence $x_{0}\not\in\operatorname*{int}D(T)$
and $z\not\in\operatorname*{int}T_{\overline{D(T)}}x_{0}=\cup_{h>0}\frac{1}{h}(\operatorname*{int}D(T)-x_{0})$
(see e.g. \cite[Proposition\ 7, p. 169]{MR749753}), where $T_{\overline{D(T)}}x_{0}$
stands for the tangent cone to the closed convex set $\overline{D(T)}$
at $x_{0}$. Therefore $[\sigma_{C}<0]\cap\operatorname*{int}T_{\overline{D(T)}}x_{0}=\emptyset$.

Note that $[\sigma_{C}<0]$ is a convex cone whose closure $[\sigma_{C}\le0]\subset T_{\overline{D(T)}}x_{0}$.
Indeed, using $Tx_{0}=Tx_{0}+N_{\overline{D(T)}}x_{0}$, every $z\in[\sigma_{C}\le0]$
satisfies $\langle z,n^{*}\rangle\le0$, for every $n^{*}\in N_{\overline{D(T)}}x_{0}$,
that is, $z\in(N_{\overline{D(T)}}x_{0})^{-}=T_{\overline{D(T)}}x_{0}$
(see e.g. \cite[Proposition\ 4, p. 168]{MR749753}).

Assume that for some $u_{0}^{*},v^{*}\in\mathbb{R}^{d}$, $L(u_{0}^{*};v^{*}):=\{u_{0}^{*}+tv^{*}\mid t\in\mathbb{R}\}\subset Tx_{0}$.
Since $T$ is monotone for every $u\in D(T)$, $u^{*}\in Tu$, $t\in\mathbb{R}$,
we find $\langle x_{0}-u,u_{0}^{*}+tv^{*}-u^{*}\rangle_{\mathbb{R}^{d}}\ge0$;
whence $\langle x_{0}-u,v^{*}\rangle_{\mathbb{R}^{d}}=0$. Since $x_{0}-D(T)$
has a non-empty interior this yields that $v^{*}=0$ and so $L(u_{0}^{*};v^{*})$
is a singleton. Therefore $C$ does not contain lines. 

According to \cite[Corollary\ 13.4.2, p. 118]{MR1451876}, $\operatorname*{int}(\operatorname*{dom}\sigma_{C})\neq\emptyset$
and so $\sigma_{C}$ is continuous on $\operatorname*{int}(\operatorname*{dom}\sigma_{C})$.
Since $[\sigma_{C}<0]$ is nonempty we find that $\operatorname*{int}([\sigma_{C}<0])\neq\emptyset$.
Indeed let $x\in[\sigma_{C}<0]$ and $y\in\operatorname*{int}(\operatorname*{dom}\sigma_{C})$.
Then for some $0\le t<1$, $x_{t}:=tx+(1-t)y\in[\sigma_{C}<0]\cap\operatorname*{int}(\operatorname*{dom}\sigma_{C})$.
Since $\sigma_{C}$ is continuous at $x_{t}$ it follows that $x_{t}\in\operatorname*{int}([\sigma_{C}<0])$.
This yields the contradiction $[\sigma_{C}<0]\cap\operatorname*{int}T_{\overline{D(T)}}x_{0}\neq\emptyset$.
Therefore $x_{0}^{*}\in Tx_{0}$ and $T\in\mathscr{M}(\mathbb{R}^{d})$.
\end{proof}

\strut

Condition (iii) in the previous theorem is equivalent to $T=T+N_{\overline{D(T)}}$.

\begin{theorem} \label{c-1dim-ic-fin dim}Let $T\in\mathcal{M}(\mathbb{R}^{d})$
be such that $^{i}D(T)\neq\emptyset$. Then $T\in\mathscr{M}(\mathbb{R}^{d})$
iff

\emph{(i)} $T+N_{L}\in\mathscr{M}(\mathbb{R}^{d})$ for every line
$L\subset\operatorname*{aff}D(T)$ such that $L\cap{}^{i}D(T)\neq\emptyset$,

\emph{(ii)} $T$ has closed convex values,

\emph{(iii)} $T=T+N_{\overline{D(T)}}$.\end{theorem}

\begin{proof} The direct implication is trivial while for the converse
assume, without loss of generality, that $0\in F:=\operatorname*{aff}D(T)$.

Conditions (i), (ii), (iii) are transmitted from $T$ to $T_{F}$
via the fact that for every line $L\subset F$, $T+N_{L}+N_{F}=T+N_{L}\in\mathscr{M}(\mathbb{R}^{d})$
and so by Lemma \ref{sub}, $(T+N_{L})_{F}=T_{F}+N_{L}\in\mathscr{M}(F)$
(argument already seen in the proof of Proposition \ref{1-dim-convex}),
(iii) is easily checked to be hereditary from $T$ to $T_{F}$, and
$T_{F}$ has closed convex values because so has $T$, $F$ is finite
dimensional, and (iii) holds (see also Proposition \ref{val-inchise}).
According to Theorem \ref{c-1dim-int-fin dim}, $T_{F}\in\mathscr{M}(F)$.

Again (iii) and $N_{F}\subset N_{\overline{D(T)}}$ provide $T=T+N_{F}$.
According to Lemma \ref{sat} $T\in\mathscr{M}(\mathbb{R}^{d})$.
\end{proof}

\strut

Theorem \ref{c-1dim-int-fin dim} allows us to re-demonstrate Löhne's
characterization of maximal monotonicity for finite-dimensional operators
(see \cite[Theorem 3.4]{MR2465513}).

\begin{theorem}\label{Lohne} An operator $T\in\mathscr{M}(\mathbb{R}^{d})$
iff \emph{(i)} $T\in\mathcal{M}(\mathbb{R}^{d})$; \emph{(ii)} $D(T)$
is nearly convex, that is, there is a convex set $C$ such that $C\subset D(T)\subset\overline{C}$;
\emph{(iii)} $T$ has convex values; \emph{(iv)} $T=T+N_{D(T)}$;
\emph{(v)} $T$ is closed. \end{theorem}

More precisely, \cite[Theorem 3.4 (iv)]{MR2465513} states that $(Tx)_{\infty}=N_{D(T)}x$,
for every $x\in D(T)$, where $(Tx)_{\infty}$ stands for the recession
cone of $Tx$. Since $Tx+(Tx)_{\infty}=Tx$, it is straightforward
that \cite[Theorem 3.4 (iv)]{MR2465513} implies our condition (iv).
Therefore Theorem \ref{Lohne} has a slightly weaker assumption (iv)
than \cite[Theorem 3.4]{MR2465513}. 

Theorem \ref{Lohne} is trivial for a singleton domain operator. The
following characterization of maximal monotonicity for $1-$dimensional
operators is a simplified version and is used in the proof of Theorem
\ref{Lohne}.

\begin{theorem}\label{c-1dim-lim} Let $U:\mathbb{R}\rightrightarrows\mathbb{R}$
be such that $\operatorname*{int}D(U)=:(\alpha,\omega)\neq\emptyset$.
Consider the assumption: (F) $U(\alpha)=U(\alpha)+\mathbb{R}_{-}$
whenever $\alpha\in D(U)$, $U(\omega)=U(\omega)+\mathbb{R}_{+}$
whenever $\omega\in D(U)$.

TFAE:

\noindent \emph{(i)} $U\in\mathscr{M}(\mathbb{R})$,

\medskip

\noindent \emph{(ii)} $U\in\mathcal{M}(\mathbb{R})$, $U$ has convex
values, (F) holds, and $U$ is closed,

\medskip

\noindent \emph{(iii)} $U\in\mathcal{M}(\mathbb{R})$, $U$ has closed
convex values, (F) holds, and\begin{equation}
\inf U((t,+\infty))\le\sup U(t),\ \forall t\in[\alpha,\omega)\cap\mathbb{R},\label{liminf}\end{equation}
 \begin{equation}
\sup U((-\infty,t))\ge\inf U(t)\ \forall t\in(\alpha,\omega]\cap\mathbb{R}.\label{limsup}\end{equation}

Here $U(A):=\cup_{x\in A}U(x)$, for $A\subset\mathbb{R}$. In all
these cases $\overline{D(U)}=\operatorname*{cl}(\operatorname*{int}D(U))=[\alpha,\omega]\cap\mathbb{R}$
and (F) is equivalent to $U=U+N_{D(U)}$. \end{theorem}

\begin{remark} Note that condition $U(\alpha)=U(\alpha)+\mathbb{R}_{-}$
whenever $\alpha\in D(U)$ comes to $U(\alpha)=(-\infty,\sup U(\alpha)]$
while $U(\omega)=U(\omega)+\mathbb{R}_{+}$ whenever $\omega\in D(U)$
is equivalent to $U(\omega)=[\inf U(\omega),+\infty)$. Therefore
in (ii) (respectively (iii)) it suffices for $U$ to have (closed)
convex values only on $ $$\operatorname*{int}D(U)$. 

For $U\in\mathcal{M}(\mathbb{R})$ relations (\ref{liminf}), (\ref{limsup})
represent a simplification of the closedness condition for $U$ and
can be equivalently restated as equalities, namely \[
\lim_{r\downarrow t}(\inf U(r))=\inf_{r>t}(\inf U(r))=\inf U((t,+\infty))=\sup U(t),\ \forall t\in[\alpha,\omega)\cap\mathbb{R},\]
\[
\lim_{\ell\uparrow t}(\sup U(\ell))=\sup_{\ell<t}(\sup U(\ell))=\sup U((-\infty,t))=\inf U(t),\ \forall t\in(\alpha,\omega]\cap\mathbb{R},\]
 since $t\rightarrow\inf U(t)$, $t\rightarrow\sup U(t)$ are non-decreasing.

More interestingly, (\ref{liminf}) is equivalent to\[
\lim_{r\downarrow t}(\sup U(r))=\inf_{r>t}(\sup U(r))=\sup U(t),\ \forall t\in[\alpha,\omega)\cap\mathbb{R},\]
while \[
\lim_{\ell\uparrow t}(\inf U(\ell))=\sup_{\ell<t}(\inf U(\ell))=\inf U(t),\ \forall t\in(\alpha,\omega]\cap\mathbb{R},\]
is an equivalent reformulation of (\ref{limsup}). 

Also, notice that (\ref{liminf}) for $t=\alpha$ spells $\inf R(U)=-\infty$
when $\alpha\not\in D(U)$ or $\sup U(\alpha)=\inf U((\alpha,+\infty))$
when $\alpha\in D(U)$. Similarly (\ref{limsup}) for $t=\omega$
spells $\sup R(U)=+\infty$ when $\omega\not\in D(U)$ or $\inf U(\omega)=\sup U((-\infty,\omega))$
when $\omega\in D(U)$.\end{remark} 

\begin{proof} (i) $\Rightarrow$ (ii) Condition $U=U+N_{D(U)}$ together
with $N_{(\alpha,\omega)}(\alpha)=\mathbb{R}_{-}$, $N_{(\alpha,\omega)}(\omega)=\mathbb{R}_{+}$(whenever
$\alpha,\omega$ are finite) provide (F). The other properties in
(ii) are usual for $U\in\mathscr{M}(\mathbb{R})$.  

(ii) $\Rightarrow$ (iii) It suffices to verify (\ref{liminf}), (\ref{limsup}).
Note that $Ux$ is bounded, for every $x\in\operatorname*{int}D(U)$,
whenever $U\in\mathcal{M}(\mathbb{R})$. Clearly, (\ref{liminf})
holds when $\inf U((t,\infty))=-\infty$. Otherwise, for every $t\in[\alpha,\omega)\cap\mathbb{R}$,
$\inf U((t,\infty))$ is finite. Take $(r_{n})_{n}\subset(\alpha,\omega)$,
$r_{n}\downarrow t$, $s_{n}:=\min U(r_{n})$ (attained since $U(r_{n})$
is a closed bounded interval) such that $\inf U((t,\infty))=\lim_{n\rightarrow\infty}s_{n}=:s\in\mathbb{R}$.
Since $U$ is closed, $s\in U(t)$ and so $s\le\sup U(t)$, that is,
(\ref{liminf}) holds. Note parenthetically that (\ref{liminf}) holds
for $t=\omega\in D(U)$ since $U(\omega)=U(\omega)+\mathbb{R}_{+}$
and $\sup U(\omega)=+\infty$. Relation (\ref{limsup}) is verified
similarly.

(iii) $\Rightarrow$ (i) First we show that $\omega<\infty$ implies
$\sup R(U)=+\infty$. Indeed, assume that $\omega$ is finite. If
$\omega\in D(U)$ from $U(\omega)=U(\omega)+\mathbb{R}_{+}$ it is
clear that $\sup R(U)=+\infty$. If $\omega\not\in D(U)$ then, according
to (\ref{limsup}), $\sup U((-\infty,\omega))=\sup R(U)\ge\inf U(\omega)=\inf\emptyset=+\infty$.
Similarly $\alpha>-\infty$ implies $\inf R(U)=-\infty$.

Let $(t,s)$ be m.r. to $U$. If $t>\omega$ or $t=\omega\not\in D(U)$
then $\omega$ is finite, and by the previous argument we find the
contradiction $s\ge\sup R(U)=+\infty$. Therefore $t<\omega$ or $t=\omega\in D(U)$.
Similarly, $t<\alpha$ or $t=\alpha\not\in D(U)$ is impossible which
leads to $t>\alpha$ or $t=\alpha\in D(U)$. Hence $t\in D(U)\cap[\alpha,\omega]$.
In particular, since $U\in\mathcal{M}(\mathbb{R})$, we also get that
$\overline{D(U)}=[\alpha,\omega]\cap\mathbb{R}$. 

If $t\in(\alpha,\omega)$ then, due to (\ref{liminf}), (\ref{limsup}),
we get $\inf U(t)\le\sup U((-\infty,t))\le s\le\inf U((t,\infty))\le\sup U(t)$.
This yields $s\in U(t)$, since $U(t)$ is a closed interval. If $t=\alpha\in D(U)$
then, according to (\ref{liminf}), $s\le\inf U((\alpha,\infty))\le\sup U(\alpha)$;
whence $s\in U(\alpha)$, since $U(\alpha)=(-\infty,\sup U(\alpha)]$.
Similarly, $t=\omega\in D(U)$ provides $s\in U(\omega)$. The proof
is complete. \end{proof}

\strut

\begin{proof}[Proof of Theorem 19] If $T\in\mathscr{M}(\mathbb{R}^{d})$
then (iii), (iv), (v) are usual properties mainly due to the preservation
of monotonicity. Since every finite-dimensional convex set has a non-empty
convex relative (algebraic) interior which is dense in the set (see
e.g. \cite[Proposition\ 2.40,\ p.\ 64]{MR1491362}), for (ii) notice
that $\overline{D(T)}$ is convex and $\emptyset\neq C:=\operatorname*{ri}\overline{D(T)}\subset D(T)\subset\overline{C}=\overline{D(T)}$
(see also \cite[Theorems 12.37, 12.41]{MR1491362}). 

For the converse note first that (ii) provide $^{i}D(T)={}^{i}C=\operatorname*{ri}C=\operatorname*{ri}D(T)$,
$\overline{D(T)}=\overline{C}$ are non-empty convex and $T$ has
closed convex values due to (iii) and (v).

Assume first that $\operatorname*{aff}D(T)=\mathbb{R}^{d}$. Hence
$^{i}D(T)=\operatorname*{int}D(T)$. To conclude it suffices to check
condition (i) in Theorem \ref{c-1dim-int-fin dim}, that is, $T+N_{L}\in\mathscr{M}(\mathbb{R}^{d})$
for every line $L\subset\mathbb{R}^{d}$ such that $L\cap\operatorname*{int}D(T)\neq\emptyset$.
According to Lemma \ref{red-fin dim} (ii), we need to show that $\mathscr{T}_{L,z}\in\mathscr{M}(\mathbb{R})$
for some $z\in\operatorname*{int}D(T)\cap L$. To this end we prove
that $U:=\mathscr{T}_{L,z}$ verifies the conditions of Theorem \ref{c-1dim-lim}
(ii).

Recall that $U=J^{*}T_{z}J$, where $L=L(z,v)$, $Jt=tv$, $t\in\mathbb{R}$
and $D(U)=J^{-1}((D(T)-z)\cap\mathbb{R}v)$. Therefore $U$ has convex
values since so does $T$. Because $z\in\operatorname*{int}C$ we
know that $\overline{(C-z)\cap\mathbb{R}v}=(\overline{C}-z)\cap\mathbb{R}v$.
This yields that $(D(T)-z)\cap\mathbb{R}v$ is nearly convex and connected.
Since $J:\mathbb{R}\rightarrow\mathbb{R}v$ is an isomorphism, $D(U)$
is a non-degenerate interval (and it is not a singleton because it
contains the non-empty open set $J^{-1}(\operatorname*{int}C-z)$).

Let $(\alpha,\omega):=\operatorname*{int}D(U)$. Then $\alpha<0$,
$\omega>0$ and, whenever $\alpha,\omega$ are finite, $z+\alpha v,z+\omega v\in\overline{D(T)}\setminus\operatorname*{int}D(T)$.
Assume that $\alpha\in D(U)$, i.e., $z+\alpha v\in D(T)$. Take $n^{*}\in N_{D(T)}(z+\alpha v)$
such that $\langle z+\alpha v-y,n^{*}\rangle>0$, for every $y\in\operatorname*{int}D(T)$.
Pick $y=z+\tfrac{\alpha}{2}v$ to get $\langle v,n^{*}\rangle<0$.
For every $s\in U(\alpha)$ there is $x^{*}\in T(z+\alpha v)$ such
that $s=\langle v,x^{*}\rangle$. For every $\lambda<0$ let $t=\lambda/\langle v,n^{*}\rangle>0$.
Then $x^{*}+tn^{*}\in T(z+\alpha v)$ due to (iv). Therefore $s+\lambda\in U(\alpha)$
and so $U(\alpha)=U(\alpha)+\mathbb{R}_{-}$ . The second part of
(F) is proved similarly.

Consider $s_{n}\in Ut_{n}$, i.e., $s_{n}=\langle v,x_{n}^{*}\rangle$
for some $x_{n}^{*}\in T(z+t_{n}v)$, $n\ge1$, with $\lim_{n\rightarrow\infty}(s_{n},t_{n})=(s,t)$.
From the local boundedness of $T$ at $z$ (see e.g. \cite[Theorem\ 2]{MR995141}),
there are $M,r>0$ such that $z+rB\subset D(T)$ and $\|x^{*}\|\le M$
for every $x^{*}\in T(z+ru)$, $u\in B:=\{x\in\mathbb{R}^{d}\mid\|x\|\le1\}$.
The monotonicity of $T$ provides $\langle t_{n}v-ru,x_{n}^{*}-x^{*}\rangle\ge0$,
for every $x^{*}\in T(z+ru)$, $u\in B$, $n\ge1$. This yields $r\|x_{n}^{*}\|\le M(|t_{n}|\|v\|+r)+t_{n}s_{n}$,
$n\ge1$; whence $(x_{n}^{*})_{n}$ is bounded. On a subsequence,
denoted for simplicity by the same index, $x_{n}^{*}\rightarrow x_{0}^{*}$
in $\mathbb{R}^{d}$ and $x_{0}^{*}\in T(z+tv)$ since $T$ is closed.
Passing to limit in $s_{n}=\langle v,x_{n}^{*}\rangle$ we find $s=\langle v,x_{0}^{*}\rangle$,
that is, $s\in Ut$ and so $U$ is closed. 

According to Theorem \ref{c-1dim-lim} (ii), $\mathscr{T}_{L,z}\in\mathscr{M}(\mathbb{R})$
and by the above argument $T\in\mathscr{M}(\mathbb{R}^{d})$.

In the general case we may assume that $0\in F:=\operatorname*{aff}D(T)$.
Then $T_{F}$ fulfills trivially (i), (ii), (iii), (iv) is due to
$(N_{D(T)})_{F}$ being the normal cone to $D(T)\subset F$, and (v)
follows from Corollary \ref{swr} (ii). Thus $T_{F}\in\mathscr{M}(F)$
followed by $T\in\mathscr{M}(\mathbb{R}^{d})$ (recall that (iv) yields
$T=T+N_{F}$). \end{proof}

\section{Line-plane characterizations}

Since the previous section provided several characterizations for
the maximal monotonicity of operators defined in $\mathbb{R}$, the
Banach spaces considered in this section are assumed to have dimension
greater than one.

\begin{proposition} \label{lp}Let $X$ be a Banach space and let
$T\in\mathcal{M}(X)$ be such that $^{ic}D(T)\neq\emptyset$. If $T+N_{P}\in\mathscr{M}(X)$
for every plane $P$ with $P\cap{}^{ic}D(T)\neq\emptyset$ then $T+N_{L}\in\mathscr{M}(X)$
for every line $L$ with $L\cap{}^{ic}D(T)\neq\emptyset$. \end{proposition}

\begin{proof} Let $L$ be a line with with $L\cap{}^{ic}D(T)\neq\emptyset$
and $P$ be a plane with $L\subset P$. Since $L\cap{}^{ic}(D(T)\cap P)\neq\emptyset$,
by Proposition \ref{afin-fin dim}, we find that $T+N_{L}=(T+N_{P})+N_{L}\in\mathscr{M}(X)$.
\end{proof}

\begin{theorem} \label{c2-ic}Let $X$ be a Banach space and let
$T\in\mathcal{M}(X)$ be such that $^{ic}D(T)\neq\emptyset$. Then
$T\in\mathscr{M}(X)$ iff

\emph{(i)} $T+N_{P}\in\mathscr{M}(X)$ for every plane $P\subset X$
such that $P\cap{}^{ic}D(T)\neq\emptyset$,

\emph{(ii)} $T$ has $w^{*}-$closed convex values.

\end{theorem}

Due to Proposition \ref{lp} condition (i) in the previous theorem
can be restated as $T+N_{A}$ is maximal monotone, for every affine
set $A$ generated by at most two linearly independent vectors and
such that $A\cap{}^{ic}D(T)\neq\emptyset$. Note that this latter
condition does not require a dimensional restriction on $X$.

\strut

\begin{proof} For the direct implication (i) follows from Proposition
\ref{afin-fin dim} and (ii) is usual for $T\in\mathscr{M}(X)$. 

For the converse let $(x_{0},x_{0}^{*})$ be m.r. to $T$ and let
$A=A(x_{0};v_{1},v_{2})$ be any affine set through $x_{0}$ generated
by $v_{1},v_{2}$ that cuts $^{ic}D(T)$. It is easily checked that
$(x_{0},x_{0}^{*})$ is m.r. to $T+N_{A}$, since $n^{*}\in N_{A}(x)$
iff $x\in A$ and $\langle v_{1},n^{*}\rangle=\langle v_{2},n^{*}\rangle=0$.
But $T+N_{A}\in\mathscr{M}(X)$.

Therefore, if $(x_{0},x_{0}^{*})$ is m.r. to $T$ then $x_{0}\in D(T)$
and\begin{equation}
\forall v_{1},v_{2}:A(x_{0};v_{1},v_{2})\cap{}^{ic}D(T)\neq\emptyset,\exists x^{*}\in Tx_{0}:\langle v_{i},x_{0}^{*}\rangle=\langle v_{i},x^{*}\rangle,i=1,2.\label{mrt-plan}\end{equation}

Assume that $x_{0}^{*}\not\in Tx_{0}$. Since $Tx_{0}^{*}$ is $w^{*}-$closed
convex, by a separation argument there is $v_{1}\in X$ such that\begin{equation}
\langle v_{1},x_{0}^{*}\rangle>\sup\{\langle v_{1},y^{*}\rangle\mid\ y^{*}\in Tx_{0}\}.\label{separare}\end{equation}

Let $v_{2}$ be such that $A:=A(x_{0};v_{1},v_{2})$ cuts $^{ic}D(T)$
and so $T+N_{A}\in\mathscr{M}(X)$. By (\ref{mrt-plan}), this yields
an $x^{*}\in Tx_{0}$ such that $\langle v_{1},x_{0}^{*}\rangle=\langle v_{1},x^{*}\rangle$
contrary to (\ref{separare}). This contradiction comes from the assumption
that $x_{0}^{*}\not\in Tx_{0}$. Hence $x_{0}^{*}\in Tx_{0}$ and
$T\in\mathscr{M}(X)$. \end{proof}

\begin{remark} After looking at Theorems \ref{c-1dim-ic-fin dim},
\ref{c2-ic} (see also Theorem \ref{c-1dim-ic} below) one wonders
whether condition (iii) in Theorem \ref{c-1dim-ic-fin dim} is necessary;
in other words whether condition (i) in Theorem \ref{c2-ic} could
be relaxed by changing the planes with lines.

Consider $C:=\{(x,y)\in\mathbb{R}^{2}\mid x\le0,y\le0\}$ and $T:D(T):=C\subset\mathbb{R}^{2}\rightrightarrows\mathbb{R}^{2}$,
$T(x,y)=N_{C}(x,y)$, for every $(x,y)\in C\setminus\{(0,0)\}$ and
$T(0,0)=\mathbb{R}_{+}(1,1)$. Then $T\not\in\mathscr{M}(\mathbb{R}^{2})$
since $T\subsetneq N_{C}$ (more precisely $T(0,0)\subsetneq N_{C}(0,0)=\{(x,y)\in\mathbb{R}^{2}\mid x\ge0,y\ge0\}=\mathbb{R}_{+}(1,0)+\mathbb{R}_{+}(0,1)$),
$T$ has closed convex values, and $\operatorname*{int}D(T)\neq\emptyset$. 

Moreover, $T+N_{L}=N_{C}+N_{L}\in\mathscr{M}(\mathbb{R}^{2})$ (see
e.g. Proposition \ref{afin-int}), for every line $L$ that cuts $\operatorname*{int}D(T)$,
that is, condition (i) in Theorem \ref{c-1dim-ic-fin dim} is fulfilled.
Indeed, this is straightforward if $L$ does not pass through the
origin. If $L$ passes through the origin note that $L$ cuts $\operatorname*{int}D(T)$
iff $L$ is non-vertical and has a positive-slope, that is, $L:=\mathbb{R}(1,m)$,
for some $m>0$. Then $N_{L}(0,0)=L^{\perp}=\mathbb{R}(-m,1)$ and
$(T+N_{L})(0,0)=\mathbb{R}_{+}(1,1)+\mathbb{R}(-m,1)=\mathbb{R}_{+}(1,0)+\mathbb{R}_{+}(0,1)+\mathbb{R}(-m,1)=(N_{C}+N_{L})(0,0)$,
since it is easily verifiable that $\mathbb{R}_{+}(1,0)\cup\mathbb{R}_{+}(0,1)\subset\mathbb{R}_{+}(1,1)+\mathbb{R}(-m,1)$.
This suffices in order to have $T+N_{L}=N_{C}+N_{L}$. 

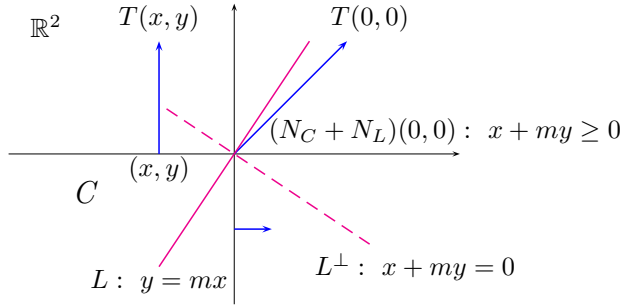
{\begin{pspicture}(0,-3)(4,2) 
\rput(3,1.2){$\mathbb{R}^2$}
\rput(5.5,-.5){\psaxes[linewidth=0.01,labels=none,ticks=none]{->}(0,0)(-3,-2)(3,2)} 
\rput(3.5,-1){\bf{\it \large C}} 
\psline[linecolor=blue,linewidth=0.02]{->}(5.5,-.5)(7,1)
\rput(7.3,1.3){$T(0,0)$} 
\psline[linecolor=blue,linewidth=0.02]{->}(4.5,-.5)(4.5,1)
\rput(4.5,1.3){$T(x,y)$} 
\rput(4.5,-.7){$(x,y)$} 
\psline[linecolor=blue,linewidth=0.02]{->}(5.5,-1.5)(6,-1.5)
\psline[linecolor=magenta,linewidth=0.02]{-}(4.5,-2)(6.5,1)
\rput(4.5,-2.2){$L:\ y=mx$} 
\psline[linecolor=magenta,linewidth=0.02,linestyle=dashed]{-}(4.6,.1)(7.3,-1.7)
\rput(7.9,-2){$L^{\perp}:\ x+my=0$}
\rput(8.3,-.2){$(N_{C}+N_{L})(0,0):\ x+my\ge 0$}
\end{pspicture}  } 

This example shows that condition (iii) in Theorem \ref{c-1dim-ic-fin dim}
is essential and, moreover, it cannot be relaxed to $T=T+N_{\operatorname*{aff}D(T)}$.
However, if the lines in Theorem \ref{c-1dim-ic-fin dim} (i) are
upgraded to planes then (iii) can be relaxed to $T=T+N_{\operatorname*{aff}D(T)}$.\end{remark}

\begin{theorem} \label{c2-ic-sat}Let $X$ be a Banach space and
let $T\in\mathcal{M}(X)$ be such that $^{ic}D(T)\neq\emptyset$.
Then $T\in\mathscr{M}(X)$ iff

\emph{(i)} $T+N_{P}\in\mathscr{M}(X)$ for every plane $P\subset\operatorname*{aff}D(T)$
such that $P\cap{}^{ic}D(T)\neq\emptyset$,

\emph{(ii)} $T$ has $w^{*}-$closed convex values,

\emph{(iii)} $T=T+N_{\operatorname*{aff}D(T)}$\emph{.}

\end{theorem}

\begin{proof} The direct implication is straightforward. For the
converse assume without loss of generality that $0\in\operatorname*{aff}D(T)=:F$.
Apply the converse of Theorem \ref{c2-ic} for $T_{F}:F\rightrightarrows F^{*}$.
For every plane $P\subset F$ such that $P\cap{}^{ic}D(T)\neq\emptyset$,
$T+N_{P}+N_{F}=T+N_{P}\in\mathscr{M}(X)$. From Lemma \ref{sub} we
see that $(T+N_{P})_{F}=T_{F}+(N_{P})_{F}\in\mathscr{M}(F)$, where
$(N_{P})_{F}\in\mathscr{M}(F)$ is the normal cone to $P\subset F$.
Hence $T_{F}\in\mathscr{M}(F)$ since, by Proposition \ref{val-inchise},
$T_{F}$ has $w^{*}-$closed values in $F^{*}$. Together with (iii)
this yields $T\in\mathscr{M}(X)$ (see Lemma \ref{sat}). \end{proof}

\strut

The natural question whether the planes in condition (i) of Theorems
\ref{c2-ic}, \ref{c2-ic-sat} can be replaced by lines is answered
in the next result; thereby providing an extension to the infinite-dimensional
context for Theorem \ref{c-1dim-ic-fin dim}.

\begin{theorem} \label{c-1dim-ic}Let $X$ be a Banach space and
let $T\in\mathcal{M}(X)$ be such that $^{ic}D(T)\neq\emptyset$.
Then $T\in\mathscr{M}(X)$ iff

\emph{(i)} $T+N_{L}\in\mathscr{M}(X)$ for every line $L\subset\operatorname*{aff}D(T)$
such that $L\cap{}^{ic}D(T)\neq\emptyset$,

\emph{(ii)} $T$ has $w^{*}-$closed convex values,

\emph{(iii)} $T=T+N_{D(T)}$\emph{.}

\end{theorem}

\begin{proof} The direct implication is clear. For the converse note
that from (iii), Lemma \ref{sat}, and Proposition \ref{val-inchise}
we may assume without loss of generality that $\operatorname*{aff}D(T)=X$
(otherwise we replace $T$ by $T_{F}$ where $0\in F:=\operatorname*{aff}D(T)$);
whence $^{ic}D(T)=\operatorname*{core}D(T)$. From (i) via Proposition
\ref{1-dim-convex} we know that $\operatorname*{core}D(T)=\operatorname*{int}D(T)$,
$\overline{D(T)}=\operatorname*{cl}(\operatorname*{int}D(T))$ are
non-empty convex; in particular $D(T)$ is nearly-convex. 

According to Theorem \ref{c2-ic}, it is enough to prove that $T+N_{P}\in\mathscr{M}(X)$
for every plane $P$ with $P\cap\operatorname*{int}D(T)\neq\emptyset$.
From Lemma \ref{red-fin dim} (ii) we know that $T+N_{P}\in\mathscr{M}(X)$
iff $\mathscr{T}_{P,z}\in\mathscr{M}(\mathbb{R}^{2})$ for some $z\in P\cap\operatorname*{int}D(T)$.

For this choice of $P$ and $z$ we plan to use Theorem \ref{c-1dim-int-fin dim}
for $\mathscr{T}_{P,z}$.

According to (ii) and Proposition \ref{cl-val-graph} (i), ${\mathscr{T}}_{P,z}$
has closed convex values.

In this case recall (\ref{icfd}), i.e., \begin{equation}
\operatorname*{int}D({\mathscr{T}}_{P,z})=J^{-1}(\operatorname*{int}D(T)-z),\ \operatorname*{cl}D({\mathscr{T}}_{P,z})=J^{-1}(\overline{D(T)}-z),\label{icp}\end{equation}
where $P=A(z;v_{1},v_{2})$, $J:\mathbb{R}^{2}\rightarrow X$, $J(t_{1},t_{2}):=t_{1}v_{1}+t_{2}v_{2}$,
and $D({\mathscr{T}}_{P,z})=J^{-1}(D(T)-z)$.

We know from (iii) and Lemma \ref{her-cone} that ${\mathscr{T}}_{P,z}\hat{t}={\mathscr{T}}_{P,z}\hat{t}+N_{D({\mathscr{T}}_{P,z})}\hat{t}$,
for every $\hat{t}=(t_{1},t_{2})\in D({\mathscr{T}}_{P,z})$. According
to Theorem \ref{c-1dim-int-fin dim}, $\mathscr{T}_{P,z}\in\mathscr{M}(\mathbb{R}^{2})$
as soon as ${\mathscr{T}}_{P,z}+N_{\ell}\in\mathscr{M}(\mathbb{R}^{2})$
for every line $\ell\subset\mathbb{R}^{2}$ such that $\ell\cap\operatorname*{int}D({\mathscr{T}}_{P,z})\neq\emptyset$.

Let $\ell\subset\mathbb{R}^{2}$ be a line such that $\ell\cap\operatorname*{int}D({\mathscr{T}}_{P,z})\neq\emptyset$.
Consider the line $L:=J(\ell)+z\subset P$ or equivalently $\ell=J^{-1}(L-z)$.
Note that $\iota_{\ell}=\iota_{L-z}\circ J$. Since $\ell\cap\operatorname*{int}D({\mathscr{T}}_{P,z})\neq\emptyset$
we know from (\ref{icp}) that $L$ cuts $\operatorname*{int}D(T)$.
We may apply the chain rule \cite[Theorem\ 2.8.3(viii), p. 123]{MR1921556}
to get $N_{\ell}=J^{*}N_{L-z}J$. This yields $\mathscr{T}_{P,z}+N_{\ell}=\mathscr{A}_{P,z}:=J^{*}A_{z}J$,
where $A:=T+N_{L}\in\mathscr{M}(X)$ due to (i). Note that $A=A+N_{P}$.
According to Lemma \ref{red-fin dim} (ii), $\mathscr{A}_{P,z}\in\mathscr{M}(\mathbb{R}^{2})$.
The proof is complete. \end{proof}

\strut

Theorems \ref{c-1dim-ic-fin dim}, \ref{c2-ic}, \ref{c2-ic-sat},
\ref{c-1dim-ic} show that under interiority conditions the sum result
contained in Theorem \ref{mmcs} (\cite[Corollary\ 4]{MR2577332})
cannot be further improved in the sense that $T$ need be maximal
monotone. 

\medskip

The advantage of a condition of type Theorem \ref{c-1dim-ic} (i)
over the (demi)closedness of the graph of an operator is that this
condition involving lines can be replaced by the hemiclosedness or
the so called {}``closedness on line'' condition, as seen in Theorem
\ref{c-1dim-lim}. 

\medskip

Let $C\subset X$ be a convex set with $\operatorname*{ri}D(T)\neq\emptyset$.
Denote by $T_{C}(x)$ the tangent cone to $C$ at $x\in X$ and by
$S_{C}(x):=\bigcup_{h>0}h(C-x)$ the cone spanned by $C-x$. Then
$\operatorname*{ri}T_{C}(x)=S_{\operatorname*{ri}C}(x)$, for very
$x\in\overline{C}$ (see \cite[Proposition\ 7, p. 169]{MR749753}).

\medskip

We are ready for the following generalization of \cite[Theorem\ 1.2]{MR0180884},
\cite[Lemma 2.2]{MR1191009}, \cite[Theorem 40.2, p. 155]{MR1723737}
(see also \cite[Remark 40.3]{MR1723737}), and \cite[Lemma 4.2]{MR2453098}.

\begin{theorem} \label{c-de-he-re}Let $X$ be a Banach space and
let $T\in\mathcal{M}(X)$ be such that $D(T)$ is nearly-convex.

Then $T\in\mathscr{M}(X)$ iff $T=T+N_{D(T)}$ and (one of) the following
assumptions hold(s):

\smallskip

\noindent \emph{(H)} $T$ has $w^{*}-$closed convex values and for
every $x\in\overline{D(T)}$, $v\in\operatorname*{ri}T_{\overline{D(T)}}(x)$\begin{equation}
\inf_{h>0}\inf_{x^{*}\in T(x+hv)}\langle v,x^{*}\rangle\le\sup_{x^{*}\in T(x)}\langle v,x^{*}\rangle.\label{hemi}\end{equation}
 \emph{(D)} $T$ has convex values and is demiclosed\emph{,}

\noindent \emph{(R)} $T$ is representable. \end{theorem}

\begin{remark} \label{hemi-equiv} Whenever $T\in\mathcal{M}(X)$,
the second part of condition \emph{(H)} in the previous theorem is
equivalent to the stronger forms: for $x\in\overline{D(T)},v\in S_{\operatorname*{ri}D(T)}(x)$

\[
\lim_{h\downarrow0}\inf_{x^{*}\in T(x+hv)}\langle v,x^{*}\rangle=\inf_{h>0}\inf_{x^{*}\in T(x+hv)}\langle v,x^{*}\rangle=\sup_{x^{*}\in T(x)}\langle v,x^{*}\rangle,\]
 \[
\lim_{h\downarrow0}\sup_{x^{*}\in T(x+hv)}\langle v,x^{*}\rangle=\inf_{h>0}\sup_{x^{*}\in T(x+hv)}\langle v,x^{*}\rangle=\sup_{x^{*}\in T(x)}\langle v,x^{*}\rangle,\]
since $h\rightarrow\inf_{x^{*}\in T(x+hv)}\langle v,x^{*}\rangle$,
$h\rightarrow\sup_{x^{*}\in T(x+hv)}\langle v,x^{*}\rangle$ are non-decreasing.
The last equivalent condition holds a striking resemblance to \cite[Lemma 2.2]{MR2465513}
(one of the key results used in the proof of \cite[Theorem 3.4]{MR2465513}). 

\end{remark}

\begin{proof} Let $T\in\mathcal{M}(X)$ be such that $D(T)$ is nearly-convex.
For every $x\in\overline{D(T)}$, $v\in S_{\operatorname*{ri}D(T)}(x)$
the line $L=L(x;v)\subset\operatorname*{aff}D(T)$ cuts $\operatorname*{ri}D(T)$
at some $z=x+t_{0}v$ with $t_{0}>0$. Let $U:={\mathscr{T}}_{L,x}$.
Then $D(U)$ is a non-degenerate interval with $0$ inside its interior
when $x\in\operatorname*{ri}D(T)$ and with $0$ as its left end-point
when $x\in\overline{D(T)}\setminus\operatorname*{ri}D(T)$ (see Lemma
\ref{int-cl-fin-dim}). Recall that $s\in U(t)$ iff $s=\langle v,x^{*}\rangle$,
for some $x^{*}\in T(x+tv)$; whence $\inf U(t)=\inf_{x^{*}\in T(x+tv)}\langle v,x^{*}\rangle$,
$\sup U(t)=\sup_{x^{*}\in T(x+tv)}\langle v,x^{*}\rangle$, and $\sup U(0)=\sup_{x^{*}\in T(x)}\langle v,x^{*}\rangle$.

It is clear that for $T\in\mathscr{M}(X)$ conditions $T=T+N_{D(T)}$,
(D), and (R) are fulfilled. According to Proposition \ref{afin-fin dim},
$T+N_{L}\in\mathscr{M}(X)$ and from Lemma \ref{red-fin dim} (ii),
$\mathscr{T}_{L,x}\in\mathscr{M}(\mathbb{R})$. Condition (H) follows
from relation (\ref{liminf}) in Theorem \ref{c-1dim-lim} (iii) for
$t=0$. 

Conversely, note that (R) $\Rightarrow$ (D) and $T$ has $w^{*}-$closed
convex values whenever (H) or (D) holds. According to Theorem \ref{c-1dim-ic},
it suffices to show that $T+N_{L}\in\mathscr{M}(X)$ for ever line
$L\subset\operatorname*{aff}D(T)$ with $L\cap\operatorname*{ri}D(T)\neq\emptyset$.
To fix notation, let $L=L(z;v)$ with $z\in\operatorname*{ri}D(T)$.
According to Lemma \ref{red-fin dim} (ii) we need to show that $U:={\mathscr{T}}_{L,z}\in\mathscr{M}(\mathbb{R})$.
To this end we use Theorem \ref{c-1dim-lim}. Both (H), (D) imply
via Proposition \ref{cl-val-graph} (i) that $U$ has closed convex
values. Also, $\operatorname*{int}D(U)=:(\alpha,\omega)\ni0$, because
$z\in\operatorname*{ri}D(T)$. Condition $T=T+N_{D(T)}$ transfers
to $U$ via Lemma \ref{her-cone} and it clearly yields that $U(\alpha)=U(\alpha)+\mathbb{R}_{-}$,
whenever $\alpha\in D(U)$ and $U(\omega)=U(\omega)+\mathbb{R}_{+}$,
whenever $\omega\in D(U)$, since $N_{[\alpha,\omega]}(\alpha)=\mathbb{R}_{-}$,
$N_{[\alpha,\omega]}(\omega)=\mathbb{R}_{+}$, that is, (F) holds.

If (D) is true then, according to Proposition \ref{cl-val-graph}
(ii), $U$ is closed. In this case the conditions in Theorem \ref{c-1dim-lim}
(ii) are met so $U\in\mathscr{M}(\mathbb{R})$ and we are done.

If (H) holds, to conclude by using Theorem \ref{c-1dim-lim} (iii),
it remains to prove that (\ref{liminf}), (\ref{limsup}) hold. 

Note that $v\in S_{\operatorname*{ri}D(T)}(z+tv)$ for $\alpha\le t<\omega$
and $-v\in S_{\operatorname*{ri}D(T)}(z+tv)$ for $\alpha<t\le\omega$;
while $z+tv\in\overline{D(T)}$, for every $t\in[\alpha,\omega]\cap\mathbb{R}$.
According to (\ref{hemi}), for every $\alpha\le t<\omega$ \[
\inf U((t,+\infty))=\inf_{h>0}\inf_{x^{*}\in T(z+tv+hv)}\langle v,x^{*}\rangle\le\sup_{x^{*}\in T(z+tv)}\langle v,x^{*}\rangle=\sup U(t).\]
 and for every $\alpha<t\le\omega$

\vspace{-.6cm}

\begin{align*}
 & \sup U((-\infty,t))=\sup_{\ell<t}\sup_{x^{*}\in T(z+\ell v)}\langle v,x^{*}\rangle\\
 & =-\inf_{h>0}\inf_{x^{*}\in T(z+tv+h(-v))}\langle-v,x^{*}\rangle\ge-\sup_{x^{*}\in T(z+tv)}\langle-v,x^{*}\rangle=\inf U(t).\end{align*}
 \end{proof}

A stronger form of condition (H) in the previous theorem has been
proven in \cite[Lemma 40.1(d)]{MR1723737} for $T\in\mathscr{M}(X)$
on $\operatorname*{int}D(T)$. Similar continuity properties of this
form are studied in our sixth section.

If $X$ is finite-dimensional then (D) together with the monotonicity
of $T$ provide (H) as seen in the proof of Theorem \ref{Lohne}.
Hence all versions of Theorem \ref{c-de-he-re} are extensions of
Theorem \ref{Lohne}. It is clear that at least for $x\in\operatorname*{int}D(T)$
the demiclosedness of $T$ implies the second part in (H) due to the
local boundedness of $T$ at $x$.  

Our next aim is to make conditions (H) and (D) in Theorem \ref{c-de-he-re}
as disjoint as possible.

\begin{theorem} \label{c-demi-fr}Let $X$ be a Banach space and
let $T\in\mathcal{M}(X)$ be such that $D(T)$ is nearly-convex. Then
$T\in\mathscr{M}(X)$ iff 

\noindent \emph{(C)} $T=T+N_{D(T)}$, 

\noindent \emph{(H')} $T$ has $w^{*}-$closed convex values and for
every $x\in\overline{D(T)}\setminus\operatorname*{ri}D(T)$, $v\in\operatorname*{ri}T_{\overline{D(T)}}(x)$\begin{equation}
\inf_{h>0}\inf_{x^{*}\in T(x+hv)}\langle v,x^{*}\rangle\le\sup_{x^{*}\in T(x)}\langle v,x^{*}\rangle,\label{demi-fr}\end{equation}

\noindent \emph{(D')} $T$ is demiclosed on $\operatorname*{ri}D(T)$.

\end{theorem}

\begin{proof} Assume first that $\operatorname*{aff}D(T)=X$. We
plan to use Theorem \ref{c-de-he-re}, therefore it suffices to show
that (\ref{hemi}) holds for $x\in\operatorname*{int}D(T)$ and for
every $v\in S_{\operatorname*{int}D(T)}(x)=X$. Assume by contradiction
that for a fixed $x\in\operatorname*{int}D(T)$ there is $v\in X$
such that (\ref{hemi}) does not hold. Taking Remark \ref{hemi-equiv}
into consideration, this entails the existence of $\epsilon_{0}>0$
such that \begin{equation}
\langle v,x^{*}\rangle\ge\langle v,y^{*}\rangle+\epsilon_{0},\ \forall h>0,\ x^{*}\in T(x+hv),\ y^{*}\in Tx.\label{contra}\end{equation}
 Take $h_{i}\downarrow0$ such that for every $i$, $x+h_{i}v$ belongs
to the neighborhood of $x$ on which $T$ is bounded. Hence any net
$(x_{i}^{*})_{i}$ with $x_{i}^{*}\in T(x+h_{i}v)$ is bounded. Eventually
on a subnet, denoted by the same index for simplicity, $x_{i}^{*}\rightarrow z^{*}\in Tx$
weakly-star in $X^{*}$, by the demiclosedness of $T$. We obtain
a contradiction if we use (\ref{contra}) for $h=h_{i}$, $x^{*}=x_{i}^{*}$,
and $y^{*}=z^{*}$ and pass to limit. 

In general we may assume that $0\in\operatorname*{aff}D(T)=:F$. The
above argument applies to $T_{F}$ via Propositions \ref{val-inchise},
\ref{demi}, because (C) provides $T=T+N_{F}$. Therefore $T_{F}\in\mathscr{M}(F)$
and $T\in\mathscr{M}(X)$ via Lemma \ref{sat}. \end{proof}

\begin{proposition} \label{NI}Let $X$ be a Banach space and let
$T\in\mathcal{M}(X)$ be such that $\operatorname*{core}D(T)\neq\emptyset$.
If $T$ is demiclosed and $Tx$ is unbounded for every $x\in D(T)\setminus\operatorname*{core}D(T)$
then $T$ is NI, $\operatorname*{int}D(T)=\operatorname*{core}D(T)$,
$\overline{D(T)}$ are convex; in particular $D(T)$ is nearly-convex.

\end{proposition}

\begin{proof}Fix $y\in\operatorname*{core}D(T)$ and $y^{*}\in Ty$.
For every $x\in X$, set $S_{x}:=\{s\in[0,1]\mid tx+(1-t)y\in D(T),\ \forall t\in[0,s]\}$,
$s_{x}:=\sup S_{x}$, and $z_{x}:=s_{x}x+(1-s_{x})y$. Then $0<s_{x}\le1$
and $z_{x}\in\overline{D(T)}\setminus\operatorname*{core}D(T)$. 

Assume by contradiction that $T$ is not NI that is $\varphi_{T}(x,x^{*})<c(x,x^{*})$
for some $(x,x^{*})\in X\times X^{*}$. Then $x\notin D(T)$ since
$D(T)\times X^{*}\subset[\varphi_{T}\ge c]$ (see e.g. \cite[Proposition\ 3.2\ (i)]{MR2389004}).
For every $s\in S_{x}$, let $t:=s/s_{x}\in(0,1]$. Then $tz_{x}+(1-t)y\in D(T)$.
Again, using $D(T)\times X^{*}\subset[\varphi_{T}\ge c]$ one gets

\vspace{-0.5cm}

\begin{align*}
 & t\varphi_{T}(z_{x},x^{*})+(1-t)c(y,y^{*})=t\varphi_{T}(z_{x},x^{*})+(1-t)\varphi_{T}(y,y^{*})\\
 & \ge\varphi_{T}(t(z_{x},x^{*})+(1-t)(y,y^{*}))=\varphi_{T}(tz_{x}+(1-t)y,tx^{*}+(1-t)y^{*})\\
 & \ge c(tz_{x}+(1-t)y,tx^{*}+(1-t)y^{*})\\
 & =tc(z_{x},x^{*})+(1-t)c(y,y^{*})-t(1-t)c(z_{x}-y,x^{*}-y^{*}).\end{align*}

This yields $\varphi_{T}(z,x^{*})\ge c(z_{x},x^{*})-(1-t)c(z_{x}-y,x^{*}-y^{*})$.
After we let $s\rightarrow s_{x}$, i.e., $t\rightarrow1$, we get
$\varphi_{T}(z_{x},x^{*})\ge c(z,x^{*})$, from which $z_{x}\neq x$.

Let $M$ be a maximal monotone extension of $T\cup\{(x,x^{*})\}$.
Then from $x\in D(M)$, $y\in\operatorname*{core}D(M)=\operatorname*{int}D(M)$,
and $z_{x}\neq x$ we know that $z_{x}\in\operatorname*{int}D(M)$.
Therefore $M$ and $T$ are locally bounded at $z_{x}$. The demiclosedness
of $T$ implies that $z_{x}\in D(T)$ so $T(z_{x})$ is non-empty
bounded in contradiction to one of our assumptions. Therefore $T$
is NI.

In this case we know that $\mathcal{R}:=[\psi_{T}=c]$ is the unique
maximal monotone extension of $T$ (see \cite[Proposition\ 4\ (iii)]{MR2594359}),
$D(\mathcal{R})\subset\operatorname*{cl}(\operatorname*{conv}D(T))$
since $\operatorname*{dom}\psi_{T}\subset\operatorname*{cl}\,\!_{s\times w^{\ast}}\operatorname*{conv}(\operatorname*{Graph}T)$,
$\operatorname*{int}D(\mathcal{R})=\operatorname*{core}D(\mathcal{R})\neq\emptyset$,
and $D(\mathcal{R})$ is nearly-convex (see \cite[Corollary\ 3]{MR2577332}
or \cite[Theorem\ 1]{MR0253014}).

Take $x\in\operatorname*{int}D(\mathcal{R})\setminus\overline{D(T)}$.
Then $z_{x}\neq x$ and $z_{x}\in\operatorname*{int}D(\mathcal{R})$
since $\operatorname*{int}D(\mathcal{R})$ is convex and contains
$y$. Therefore the operators $\mathcal{R}$ and $T$ are locally
bounded at $z_{x}$ and $z_{x}\in D(T)$ because $T$ is demiclosed.
Thus we reach at the contradiction $z_{x}\in D(T)\setminus\operatorname*{core}D(T)$
and $T(z_{x})$ is non-empty bounded.

Hence $\operatorname*{int}D(\mathcal{R})\subset\overline{D(T)}\subset\overline{D(\mathcal{R})}$;
whence $\overline{D(T)}=\overline{D(\mathcal{R})}$ is convex and
$D(\mathcal{R})\subset\overline{D(T)}$ because $D(\mathcal{R})$
is nearly-convex. Again, for every $x\in\operatorname*{int}D(\mathcal{R})$,
$\mathcal{R}$ and $T$ are locally bounded at $x$ and $x\in D(T)$
due to the demiclosedness of $T$, that is, $\operatorname*{int}D(\mathcal{R})\subset D(T)$.
This yields that $\operatorname*{int}D(T)=\operatorname*{int}D(\mathcal{R})=\operatorname*{core}D(T)$
is convex and $\overline{D(T)}=\overline{D(\mathcal{R})}=\operatorname*{cl}(\operatorname*{int}D(T))$.
\end{proof}

\begin{remark}The unboundedness of an operator $T:X\rightrightarrows X^{*}$
at every $x\in D(T)\setminus\operatorname*{core}D(T)$ does not transmit
to its restriction $T_{\operatorname*{aff}D(T),z}$ even though $\operatorname*{ri}D(T)\neq\emptyset$.
Indeed, take $F\subset X$ a proper closed linear subspace and $T=F\times F^{\perp}$.
Then $Tx=F^{\perp}$ is unbounded, for every $x\in F$ while $T_{F}=F\times\{0\}$,
i.e., $T_{F}x=\{0\}$, for $x\in F$. \end{remark}

\begin{lemma}\label{rel} Let $(X,\|\cdot\|)$ be a Banach space,
let $T\in\mathcal{M}(X)$ be such that $\operatorname*{core}D(T)\neq\emptyset$,
and let $x\in D(T)\setminus\operatorname*{core}D(T)$. Consider the
conditions

\medskip

\emph{(i)} $Tx$ is unbounded;

\medskip

\emph{(ii)} $\sup\{\langle x-y,x^{*}\rangle\mid x^{*}\in Tx\}=+\infty$,
for every (some) $y\in\operatorname*{core}D(T)$. 

\medskip

Then \emph{(ii)} $\Rightarrow$ \emph{(i)}. If, in addition, $ $$\operatorname*{int}D(T)=\operatorname*{core}D(T)$
then \emph{(i)} $\Rightarrow$ \emph{(ii)}. \end{lemma}

\begin{proof} (ii) $\Rightarrow$ (i) Let $y\in\operatorname*{core}D(T)$
be such that $\sup\{\langle x-y,x^{*}\rangle\mid x^{*}\in Tx\}=+\infty$.
Since $x\neq y$ and $\langle x-y,x^{*}\rangle\le\|x-y\|\|x^{*}\|$,
this gives $\sup\{\|x^{*}\|\mid x^{*}\in Tx\}=+\infty$. 

(i) $\Rightarrow$ (ii) For every $y\in\operatorname*{int}D(T)$ let
$M,r>0$ be such that $y+rB_{X}\in D(T)$ and $\|y^{*}\|\le M$, for
every $y^{*}\in T(y+ru)$, $u\in B_{X}$. The monotonicity of $T$
implies that $\langle x-y-ru,x^{*}-y^{*}\rangle\ge0$, for every $x^{*}\in Tx$,
$y^{*}\in T(y+ru)$, $u\in B_{X}$. This yields $\langle x-y,x^{*}\rangle\ge r\langle u,x^{*}\rangle-\|x-y-ru\|\|y^{*}\|$,
for every $x^{*}\in Tx$, $y^{*}\in T(y+ru)$, $u\in B_{X}$; followed
by $\langle x-y,x^{*}\rangle\ge r\langle u,x^{*}\rangle-(\|x-y\|+r)M$,
for every $x^{*}\in Tx$, $u\in B_{X}$. Pass to supremum over $u\in B_{X}$
to get $\langle x-y,x^{*}\rangle\ge r\|x^{*}\|-(\|x-y\|+r)M$, for
every $x^{*}\in Tx$; whence $\sup\{\langle x-y,x^{*}\rangle\mid x^{*}\in Tx\}=+\infty$.
\end{proof}

\strut

The advantage of condition (ii) in the previous lemma is that it transmits
to $T_{\operatorname*{aff}D(T),z}$; thereby allowing to relativize
Proposition \ref{NI}. 

\begin{proposition} \label{NIR}Let $X$ be a Banach space and let
$T\in\mathcal{M}(X)$ be such that $^{ic}D(T)\neq\emptyset$ and $T=T+N_{\operatorname*{aff}D(T)}$.
If $T$ is demiclosed and for every $x\in D(T)\setminus{}^{ic}D(T)$
there is $y\in{}^{ic}D(T)$ such that $\sup\{\langle x-y,x^{*}\rangle\mid x^{*}\in Tx\}=+\infty$
then $T_{\operatorname*{aff}D(T),z}$ is NI, for every $z\in\operatorname*{aff}D(T)$
and $D(T)$ is nearly-convex. \end{proposition}

\begin{proof} If $\operatorname*{aff}D(T)=X$ the conclusion follows
from Proposition \ref{NI} and Lemma \ref{rel}. In general we may
assume that $0\in F:=\operatorname*{aff}D(T)$ otherwise we replace
$T$ by $T_{z}$ with $z\in\operatorname*{aff}D(T)$. Note that $T_{F}$
is demiclosed due to $T=T+N_{F}$ and Proposition \ref{demi}. Also
$\sup\{\langle x-y,f^{*}\rangle\mid f^{*}\in T_{F}x\}=\sup\{\langle x-y,x^{*}\rangle\mid x^{*}\in Tx\}=+\infty$
which allows the conclusion for $T_{F}$. Therefore $T_{F}$ is NI
and $D(T)$ is nearly convex. \end{proof}

\strut

The following results are versions of Theorems \ref{c-de-he-re},
\ref{c-demi-fr} mainly by relaxing condition (C).

\begin{theorem} \label{u+r-mm}Let $X$ be a Banach space and $T:X\rightrightarrows X^{*}$
be such that $^{ic}D(T)\neq\emptyset$. Then $T\in\mathscr{M}(X)$
iff $T=T+N_{\operatorname*{aff}D(T)}$, $T$ is representable, and
for every $x\in D(T)\setminus{}^{ic}D(T)$ there is $y\in{}^{ic}D(T)$
such that $\sup\{\langle x-y,x^{*}\rangle\mid x^{*}\in Tx\}=+\infty$.
In this case $^{ic}D(T)=\operatorname*{ri}D(T)$ and $D(T)$ is nearly-convex.\end{theorem}

\begin{proof} For the direct implication if $T\in\mathscr{M}(X)$
then $T$ is representable with $^{ic}D(T)=\operatorname*{ri}D(T)$
and $D(T)$ nearly-convex (see \cite[Corollary\ 3]{MR2577332}) and
every $x\in D(T)\setminus\operatorname*{ri}D(T)$ is a support point
of $D(T)$, i.e., there is $u^{*}\in N_{D(T)}x$, $u^{*}\neq0$. This
makes $Tx$ unbounded because $Tx=Tx+N_{D(T)}x$ and so $\sup\{\langle x-y,x^{*}\rangle\mid x^{*}\in Tx\}=+\infty$,
for every $y\in{}^{ic}D(T)$.

Proposition \ref{NIR} may be used for the converse implications since
every representable operator is monotone demiclosed to find that $T_{F}$
is NI and $D(T)$ is nearly-convex, where we may assume without loss
of generality that $0\in F:=\operatorname*{aff}D(T)$.

Let $h\in\mathcal{R}_{T}$. Consider $g(x,f^{*})=\inf\{h(x,x^{*})\mid x^{*}|_{F}=f^{*}\}$,
$(x,f^{*})\in F\times F^{*}$. According to \cite[Theorem\ 5.1]{MR2453098},
$g\in\mathcal{R}_{T_{F}}$ since $h,h^{\square}\le\psi_{T}$ (see
\cite[Remark\ 3.6]{MR2453098} or \cite[(5)]{MR2577332}) which yields
$\operatorname*{Pr}_{X}(\operatorname*{dom}h^{*})\subset\operatorname*{Pr}_{X}(\operatorname*{dom}\psi_{T})\subset\overline{D(T)}\subset F$
because $D(T)$ is nearly convex. Hence $T_{F}$ is representable
and from \cite[Theorem\ 2.3]{MR2207807} we know that $T\in\mathscr{M}(F)$.
Again, Lemma \ref{sat} provides $T\in\mathscr{M}(X)$. \end{proof}

\strut

In particular, the previous result can be used to reprove Theorem
\ref{c-de-he-re} under the (R) assumption.

\begin{corollary} \label{ao}Let $X$ be a Banach space and $T:X\rightrightarrows X^{*}$
be such that $D(T)$ is algebraically open (that is, $D(T)={}^{ic}D(T)$).
Then $T\in\mathscr{M}(X)$ iff $T=T+N_{\operatorname*{aff}D(T)}$
and $T$ is representable. In this case $^{ic}D(T)=\operatorname*{ri}D(T)$
and $D(T)$ is nearly-convex.\end{corollary}

\begin{remark}One wonders whether, in the previous results, the contribution
of $T|_{D(T)\setminus{}^{ic}D(T)}$ could be avoided. Let $T\in\mathscr{M}(X)$
with $ $$\operatorname*{int}D(T)\neq\emptyset$. In general, would
the NI type and representability transmit from $T$ to $T|_{\operatorname*{int}D(T)}$?
The answer is negative. For example, take $C\subsetneq X$ closed
convex with $\operatorname*{int}C\neq\emptyset$. Then $N_{C}\in\mathscr{M}(X)$
(so it is NI and representable) while $N_{C}|_{\operatorname*{int}C}=\operatorname*{int}C\times\{0\}$
is neither NI (because it is not unique; $X\times\{0\}$ and $N_{C}$
being two different strict maximal monotone extensions of $N_{C}|_{\operatorname*{int}C}$)
nor representable (since otherwise, according to Corollary \ref{ao},
$N_{C}|_{\operatorname*{int}C}\in\mathscr{M}(X)$).\end{remark}

In the next result we avoid condition (C) completely.

\begin{theorem} \label{dr}Let $X$ be a Banach space and let $T:X\rightrightarrows X^{*}$
be dual-representable. Consider the conditions

\medskip

\emph{(i)} $D(T)$ is nearly-convex,

\medskip

\emph{(ii)} $^{ic}D(T)$ is non-empty convex, and $D(T)\subset\operatorname*{cl}({}^{ic}D(T))$,

\medskip

\emph{(iii)} $^{ic}\operatorname*{Pr}_{X}(\operatorname*{dom}h)\neq\emptyset$,
for some $h\in\mathcal{D}_{T}$,

\medskip

\emph{(iv)} $T\in\mathscr{M}(X)$.

\medskip

Then \emph{(i) $\Leftrightarrow$ (ii) $\Leftrightarrow$ (iii)} $\Rightarrow$
\emph{(iv)}. \end{theorem}

\begin{proof} (i) $\Rightarrow$ (ii) is straightforward.

(ii) $\Rightarrow$ (iii) Every $h\in\mathcal{D}_{T}$ has $h\le\psi_{T}$;
therefore $^{ic}D(T)\subset D(T)\subset\operatorname*{Pr}_{X}(\operatorname*{dom}h)\subset\operatorname*{Pr}_{X}(\operatorname*{dom}\psi_{T})\subset\overline{D(T)}=\operatorname*{cl}({}^{ic}D(T))$
since $\overline{D(T)}$ is convex. This implies that $\operatorname*{aff}(\operatorname*{Pr}_{X}(\operatorname*{dom}h))=\operatorname*{aff}D(T)$
and $^{ic}D(T)=^{ic}\operatorname*{Pr}_{X}(\operatorname*{dom}h)\neq\emptyset$. 

(iii) $\Rightarrow$ (iv) It suffices to prove that $T$ is NI. To
this end we show that $x\in D(T)$ whenever $z=(x,x^{*})$ is m.r.
to $T$. 

Eventually by making a translation we may assume without loss of generality
that $0\in\,^{ic}\operatorname*{Pr}_{X}(\operatorname*{dom}h)$. 

Let $f\in\Lambda(\mathbb{R}^{2})$ be given by $f(t,s)=\inf\{h(y,y^{*})\mid(y,y^{*})\in C(t,s)\}$,
where $C:\mathbb{R}^{2}\rightrightarrows Z$, $(y,y^{*})\in C(t,s)$
iff $y=tx$, $\left\langle x,y^{*}\right\rangle =s$. Note that $C$
is linear and $s_{\mathbb{R}^{2}}\times(s_{X}\times w_{X^{*}}^{*})-$closed.
Since $h\in\Gamma(Z)$ and $D:=\operatorname*{Im}C-\operatorname*{dom}h=(\mathbb{R}x-\operatorname*{Pr}_{X}(\operatorname*{dom}h))\times X^{*}$
has $\operatorname*{aff}D=(\mathbb{R}x-\operatorname*{aff}(\operatorname*{Pr}_{X}(\operatorname*{dom}h))\times X^{*}$
closed, because $\operatorname*{aff}(\operatorname*{Pr}_{X}(\operatorname*{dom}h)$
is closed and $\mathbb{R}x$ is finite dimensional, we get $0\in\,^{ic}D$.

According to \cite[Theorem 2.8.6 (v)]{MR1921556}, one gets \[
f^{*}(s,t)=\min\{h^{*}(y^{*},y^{**})\mid(s,t)\in C^{*}(y^{*},y^{**})\},\ (s,t)\in\mathbb{R}^{2}.\]
Since $(s,t)\in C^{*}(y^{*},y^{**})$ iff $y^{**}=tx$, $\left\langle x,y^{*}\right\rangle =s$
one gets $f^{\square}(t,s)=f^{*}(s,t)=\min\{h^{\square}(tx,y^{*})\mid\left\langle x,y^{*}\right\rangle =s\}$.
Note that since $f\ge c$, we have $f^{\square\square}=\operatorname*{cl}f\ge c$
and $f^{\square}\in\mathscr{D}_{[f^{\square}=c]}$. From \cite[Theorem\ 3.1]{MR1974634}
we know that $[f^{\square}=c]\in\mathscr{M}(\mathbb{R}^{2})$.

Note that from $h^{\square}\in\mathscr{R}_{T}$ one finds that for
every $(t,s)\in[f^{\square}=c]$ there is $y^{*}\in T(tx)$ such that
$\left\langle x,y^{*}\right\rangle =s$. Since $z=(x,x^{*})$ is m.r.
to $T$ and implicitly to $(tx,y^{*})$ we get $(1-t)(\left\langle x,x^{*}\right\rangle -s)=\left\langle x-tx,x^{*}-y^{*}\right\rangle \ge0$,
that is $(1,\left\langle x,x^{*}\right\rangle )$ is m.r. to $[f^{\square}=c]$.
Hence there is $y^{*}\in T(x)$ such that $\left\langle x,y^{*}\right\rangle =\left\langle x,x^{*}\right\rangle $;
in particular $x\in\operatorname*{dom}T$.

(iii) $\Rightarrow$ (iv) Whenever (iii) happens we know that $T\in\mathscr{M}(X)$.
From \cite[Corollary\ 3]{MR2577332} we have that $^{ic}\operatorname*{Pr}_{X}(\operatorname*{dom}h)={}^{ic}D(T)=\operatorname*{ri}D(T)$
and; whence $D(T)$ is nearly-convex. \end{proof}

\begin{remark}Under the assumption $0\in\,^{ic}\operatorname*{Pr}_{X}(\operatorname*{dom}h)$
for some $h\in\mathcal{D}_{T}$ the implication (iii) $\Rightarrow$
(iv) has been previously proved in \cite[Theorem\ 3.1]{MR2675662}.
The advantage of our argument for (iii) $\Rightarrow$ (iv) is that,
besides its brevity, it works for $X$ merely a locally convex space
under the modified assumption that $^{ib}\operatorname*{Pr}_{X}(\operatorname*{dom}h)\neq\emptyset$,
for some $h\in\mathcal{D}_{T}$, where for $S\subset X$, $^{ib}S={}^{i}S$
if the linear subspace parallel to $\operatorname*{aff}S$ is barrelled,
$^{ib}S=\emptyset$ otherwise.\end{remark}

\section{Continuity properties}

\begin{proposition} \label{sai}Let $(X,\|\cdot\|)$ be a normed
barrelled space and let $A\subset X^{*}$ be non-empty, $w^{*}-$closed,
and convex. Then $\operatorname*{int}(\operatorname*{dom}\sigma_{A})\neq\emptyset$
iff there exist $y\in X$, $\beta\in\mathbb{R}$ such that $\langle y,x^{*}\rangle\ge\|x^{*}\|+\beta$,
for every $x^{*}\in A$. \end{proposition}

\begin{proof} Assuming that $x_{0}\in\operatorname*{int}(\operatorname*{dom}\sigma_{A})$,
$\sigma_{A}$ is continuous at $x_{0}$; whence, for some $r>0$,
$f(u):=\sigma_{A}(x_{0}+u)\le\gamma<\infty$, for every $u\in rB_{X}$
or equivalently, $f\le I_{rB_{X}}+\gamma$. Therefore $f^{*}(x^{*})=I_{A}(x^{*})-\langle x_{0},x^{*}\rangle\ge(I_{rB_{X}}+\gamma)^{*}(x^{*})=r\|x^{*}\|-\gamma$,
for $x^{*}\in X^{*}$. This implies $\langle y,x^{*}\rangle\ge\|x^{*}\|+\beta$,
for every $x^{*}\in A$, where $y=-(1/r)x_{0}$, $\beta=-\gamma/r$. 

Conversely, assume that, for some $y\in X$, $\beta\in\mathbb{R}$,
$\langle y,x^{*}\rangle\ge\|x^{*}\|+\beta$, for every $x^{*}\in A$.
Then $h(x^{*}):=I_{A}(x^{*})+\langle y,x^{*}\rangle\ge g(x^{*}):=\|x^{*}\|+\beta$,
for every $x^{*}\in X^{*}$, followed by $h^{*}(u)=\sigma_{A}(-y+u)\le g^{*}(u)=I_{B_{X}}(u)-\beta$,
for every $u\in X$, that is, $\sigma_{A}(-y+u)\le-\beta$, for every
$u\in B_{X}$. This yields $-y\in\operatorname*{int}(\operatorname*{dom}\sigma_{A})$.
\end{proof}

\begin{lemma} \label{saz}Let $(X,\|\cdot\|)$ be a normed barrelled
space and let $A\subset X^{*}$ be non-empty, $w^{*}-$closed, and
convex with $[\sigma_{A}<0]\neq\emptyset$. Then $\operatorname*{int}[\sigma_{A}<0]\neq\emptyset$
iff $\operatorname*{int}(\operatorname*{dom}\sigma_{A})\neq\emptyset$.
\end{lemma}

\begin{proof} While the direct implication is straightforward for
the converse let $y\in\operatorname*{int}(\operatorname*{dom}\sigma_{A})$,
$x\in[\sigma_{A}<0]$, and $r,M>0$ such that $\sigma_{A}(y+ru)\le M<\infty$,
for every $u\in B_{X}$ (since $\sigma_{A}$ is continuous at $y$).
Fix $1>t>M/(M-\sigma_{A}(x))$ and let $x_{t}:=tx+(1-t)y$. We have\[
\sigma_{A}(x_{t}+(1-t)ru)=\sigma_{A}(tx+(1-t)(y+ru))\le t\sigma_{A}(x)+(1-t)M<0,\ \forall u\in B_{X},\]
which shows that $x_{t}\in\operatorname*{int}([\sigma_{A}<0])$. \end{proof} 

\begin{proposition} \label{si}Let $(X,\|\cdot\|)$ be a normed barrelled
space, let $T\in\mathcal{M}(X)$ be such that $\operatorname*{int}D(T)\neq\emptyset$,
and let $x\in X$, $x^{*}\in X^{*}$ be such that $Tx$ is $w^{*}-$closed
convex and $x^{*}\not\in Tx$. Then $\operatorname*{int}[\sigma_{Tx}<x^{*}]\neq\emptyset$.
If, in addition, $Tx=Tx+N_{D(T)}x$ and $D(T)$ is nearly-convex then
$\operatorname*{int}[\sigma_{Tx}<x^{*}]\subset\operatorname*{int}T_{D(T)}x=S_{\operatorname*{int}D(T)}(x)$.
\end{proposition}

\begin{proof} If $x\not\in D(T)$ then $\sigma_{Tx}=-\infty$, $[\sigma_{Tx}<x^{*}]=X$,
and the conclusion is trivial. If $x\in D(T)$ let $A:=Tx-x^{*}$.
A simple separation argument shows that $[\sigma_{A}<0]=[\sigma_{Tx}<x^{*}]\neq\emptyset$.
According to Lemma \ref{saz} and Proposition \ref{sai}, it suffices
to show that for some $y\in X$, $\beta\in\mathbb{R}$ one has $\langle y,y^{*}-x^{*}\rangle\ge\|y^{*}-x^{*}\|+\beta$,
for every $y^{*}\in Tx$. Equivalently, we need to show that there
exist $y\in X$, $\beta\in\mathbb{R}$ such that \begin{equation}
\langle y,y^{*}\rangle\ge\|y^{*}\|+\beta,\ \forall y^{*}\in Tx.\label{e}\end{equation}
Fix $z\in\operatorname*{int}D(T)$ and let $M,r>0$ be such that $z+rB_{X}\subset D(T)$
and $\|z^{*}\|\le M$, for every $z^{*}\in T(z+ru)$, $u\in B_{X}$.
The monotonicity of $T$ provides\[
\langle z+ru-x,z^{*}-y^{*}\rangle\ge0,\ \forall z^{*}\in T(z+ru),u\in B_{X},y^{*}\in Tx.\]
This yields $\langle x-z,y^{*}\rangle\ge r\|y^{*}\|-M(\|z-x\|+r)$,
that is, (\ref{e}) holds with $y=(1/r)(x-z)$, $\beta=-M/r(\|z-x\|+r)$. 

If, in addition, $Tx=Tx+N_{D(T)}x$ and $D(T)$ is nearly-convex then
$[\sigma_{Tx}<x^{*}]\subset(N_{D(T)}x)^{-}=T_{\overline{D(T)}}x=\operatorname*{cl}(\operatorname*{int}T_{\overline{D(T)}}x)=\operatorname*{cl}S_{\operatorname*{int}D(T)}(x)$.
Therefore $\emptyset\neq\operatorname*{int}[\sigma_{Tx}<x^{*}]\subset\operatorname*{int}(\operatorname*{cl}S_{\operatorname*{int}D(T)}(x))=S_{\operatorname*{int}D(T)}(x)$.
\end{proof} 

\begin{remark} The previous results allow us to present a different
argument for the converse of Theorem \ref{c-1dim-ic}.

Indeed, recall from Proposition \ref{1-dim-convex} that $\mathcal{R}:=[\psi_{T}=c]$
is the unique maximal monotone extension of $T$ and $D(T)=D(\mathcal{R})$
is nearly-convex. We may assume without loss of generality that $\operatorname*{aff}D(T)=X$.
It suffices to show that $T=\mathcal{R}$. Assume by contradiction
that there is $(x,x^{*})\in\mathcal{R}$ such that $x^{*}\not\in Tx\neq\emptyset$.
According to Proposition \ref{si}, condition $T=T+N_{D(T)}$ provides
$\emptyset\neq\operatorname*{int}[\sigma_{Tx}<x^{*}]\subset S_{\operatorname*{int}D(T)}(x)$.
Take $v\in\operatorname*{int}[\sigma_{Tx}<x^{*}]$. Since $v\in S_{\operatorname*{int}D(T)}(x)$
we know that $L(x,v)\cap\operatorname*{int}D(T)\neq\emptyset$, so,
as seen in the proof of Proposition \ref{1-dim-convex}, there is
$y^{*}\in Tx$ such that $\langle v,x^{*}\rangle=\langle v,y^{*}\rangle$
in contradiction with $v\in[\sigma_{Tx}<x^{*}]$. Therefore $T=\mathcal{R}\in\mathscr{M}(X)$.
\end{remark}

The following conjecture is stated in \cite[p.\ 21]{MR2362689}: -
Every maximal monotone operator with a non-empty domain interior and
defined in a Banach space is strongly$\times$ bounded weakly-star
closed. A stronger form of this conjecture, namely the closedness
property with respect to the strong$\times$weak-star topology, is
known to hold for the normal cone to a closed convex set with non-empty
interior (see e.g. \cite[Corollary on p. 58]{MR709590}). The next
results give a positive answer to the mentioned conjecture in a relaxed
context. 

\begin{theorem} \label{FB}Let $(X,\|\cdot\|)$ be a normed barrelled
space and $T\in\mathcal{M}(X)$ be such that $\operatorname*{int}D(T)\neq\emptyset$.
Let $\{(x_{i},x_{i}^{*})\}_{i\in I}\subset T$ be a net indexed on
the directed set $(I,\le)$ such that $x_{i}\rightarrow x_{0}$, strongly
in $X$. Then there exist $\gamma>0$, $\beta\in\mathbb{R}$, $i_{0}\in I$
such that the following {}``a priori'' estimate holds\begin{equation}
\langle x_{0},x_{i}^{*}\rangle\ge\gamma\|x_{i}^{*}\|+\beta\ \forall i\ge i_{0}.\label{ap}\end{equation}

As a consequence we have the following two cases

\noindent \emph{(i)} If $\limsup_{i\in I}\langle x_{0},x_{i}^{*}\rangle<\infty$
then, for some index $i'\in I$, $\{x_{i}^{*}\}_{i\ge i'}$ is bounded
in $X^{*}$. If, in addition

\emph{(a)} $T$ is demiclosed then $x_{0}\in D(T)$; 

\emph{(b)} $T$ is demiclosed and $x_{i}^{*}\rightarrow x_{0}^{*}$,
weakly-star in $X^{*}$ then $(x_{0},x_{0}^{*})\in T$.

In particular, every monotone demiclosed operator with a non-empty
domain interior is strongly$\times$weakly-star closed.

\noindent \emph{(ii) }If $\limsup_{i\in I}\langle x_{0},x_{i}^{*}\rangle=\infty$
then, at least on a subnet, $\|x_{i}^{*}\|^{-1}x_{i}^{*}\rightarrow u^{*}\in N_{D(T)}x_{0}$
weakly-star in $X^{*}$, and $\langle x_{0},u^{*}\rangle\ge\gamma$.

\end{theorem}

\begin{proof} See the published version. \end{proof}

\begin{theorem} Let $X$ be a Banach space and let $T\in\mathcal{M}(X)$
be demiclosed with $\operatorname*{ri}D(T)\neq\emptyset$ and $T=T+N_{\operatorname*{aff}D(T)}$.
Then $T$ is $s\times w^{*}-$closed in $X\times X^{*}$. \end{theorem}

\begin{proof} Assume without loss of generality that $0\in F:=\operatorname*{aff}D(T)$.
According to Proposition \ref{demi}, $T_{F}\in\mathcal{M}(F)$ is
demiclosed and $\operatorname*{int}D(T_{F})=\operatorname*{ri}D(T)\neq\emptyset$.
From Theorem \ref{FB}, $T_{F}$ is $s\times w^{*}-$closed in $F\times F^{*}$.
Proposition \ref{sxw*} completes the proof. \end{proof} 

\begin{theorem} Let $X$ be a Banach space and let $T\in\mathscr{M}(X)$
be such that $^{ic}D(T)\neq\emptyset$. Then $T$ is $s\times w^{*}-$closed
in $X\times X^{*}$. \end{theorem}

\begin{proof} Again we may assume that $0\in F:=\operatorname*{aff}D(T)$.
Since $T\in\mathscr{M}(X)$ we know that $T$ is demiclosed, $\operatorname*{ri}D(T)={}^{ic}D(T)\neq\emptyset$,
and $T=T+N_{F}$. The conclusion follows from the previous theorem.
\end{proof}

\strut

It is common knowledge that a maximal monotone operator with a non-empty
domain interior and defined in a Banach space is strongly$\times$weakly-star
upper semi-continuous on the interior of its domain (see e.g. \cite[Theorem\ 6.7, p. 55]{MR687963},
\cite[Proposition B, p. 113]{MR0425687}, or \cite[Theorem 2.5 (i), p. 155]{MR1079061}).
For completeness we include a proof of this result in a slightly relaxed
context.

\begin{theorem} Let $X$ be a normed barrelled space and let $T\in\mathcal{M}(X)$.
If $T$ is demiclosed then $T$ is $s\times w^{*}-$upper semicontinuous
at every $x\in\operatorname*{core}D(T)$. \end{theorem}

\begin{proof} Assume that $T$ is not $s\times w^{*}-$upper semicontinuous
at some $x\in\operatorname*{core}D(T)$, that is, there is a $w^{*}-$open
set $V\supset Tx$, $x_{n}\rightarrow x$, strongly in $X$, and $x_{n}^{*}\in Tx_{n}$
such that $x_{n}^{*}\not\in V$, for every $n\ge1$. For $n$ large
enough, $x_{n}\in U$, where $U$ is the neighborhood of $x$ on which
$T(U)$ is bounded. Therefore $(x_{n}^{*})_{n}$ is bounded and eventually
on a subnet, denoted by the same index for simplicity, $x_{n}^{*}\rightarrow x^{*}$
weakly-star in $X^{*}$. Then $x^{*}\in Tx$ since $T$ is demiclosed
and $x^{*}\not\in V$ because $V$ is $w^{*}-$open. This spells the
obvious contradiction $Tx\not\subset V$. \end{proof}

\strut

It is easily verifiable that the $s\times w^{*}-$upper semicontinuity
of a maximal monotone operator $T$ does not necessarily hold on the
boundary of $D(T)$ even though the context is finite dimensional.
For example take $B$ the closed unit ball in $\mathbb{R}^{2}$ endowed
with the usual euclidean inner product {}``$\langle\cdot,\cdot\rangle$''
and norm {}``$\|\!\cdot\!\|$''. Then $N_{B}$ is not upper semicontinuous
at any $x\in\mathbb{R}^{2}$ with $\|x\|=1$. Indeed, $N_{B}x=\mathbb{R}_{+}x$,
for every $\|x\|=1$ and for every $t\ge0$, $\|y\|=1$, $d:=\operatorname*{dist}(ty,\mathbb{R}_{+}x)=t\sqrt{1-\langle x,y\rangle^{2}}$.
Hence for $W:=\tfrac{1}{2}B$, $V:=N_{B}x+W$ (a neighborhood of $N_{B}x$),
and every $U$ a neighborhood of $x$ we pick $y\in U$ with $\|y\|=1$,
$y\neq x$ (so $\langle x,y\rangle\neq1$), and $t=(1-\langle x,y\rangle^{2})^{-1/2}$.
Then $ty\in N_{B}(U)$ and $ty\not\in V$ due to $\operatorname*{dist}(ty,\mathbb{R}_{+}x)=1$.

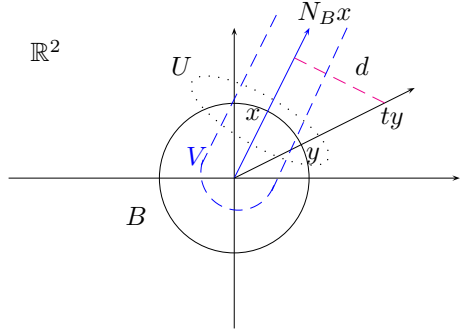
{\begin{pspicture}(0,-3)(4,2) 
\rput(3,1.2){$\mathbb{R}^2$}
\rput(5.5,-.5){\psaxes[linewidth=0.01,labels=none,ticks=none]{->}(0,0)(-3,-2)(3,2)} 
\pscircle[linewidth=0.01,dimen=outer](5.5,-.5){1}
\rput(4.2,-1){$B$} 
\psline[linecolor=blue,linewidth=0.01]{->}(5.5,-.5)(6.5,1.5)
\rput(6.7,1.7){$N_B x$} 
\rput(5.75,.35){$x$}
\psline[linecolor=black,linewidth=0.01]{->}(5.5,-.5)(7.9,.7)
\rput(6.55,-.2){$y$}
\rput(7.6,0.3){$ty$}
\psline[linestyle=dashed,linecolor=magenta,linewidth=0.01]{-}(7.5,.5)(6.3,1.1)
\rput(7.2,1){$d$}
\rput{-30.0}(1.6,1.8){\psellipse[linestyle=dotted,linewidth=0.02,dimen=outer](4.44,0.7917187)(1.05,0.34)}
\rput(4.8,1){$U$}
\rput{160.0}(11.45,-4.7){\psarc[linewidth=0.01,linecolor=blue,linestyle=dashed](7,-2){.5}{0.0}{180.0}}
\psline[linestyle=dashed,linecolor=blue,linewidth=0.01]{-}(5.06,-.33)(6.06,1.67)
\psline[linestyle=dashed,linecolor=blue,linewidth=0.01]{-}(6,-.67)(7,1.5)
\rput(5,-.2){\color{blue}{$V$}}
\end{pspicture}  } 

However, in a finite dimensional settings a different form of the
$s\times w^{*}-$upper semicontinuity, namely, the (Q) (or Cesari)
property holds for maximal monotone operators (see e.g. \cite[Lemma\ 3.2]{MR2465513},
\cite{MR687963}). Our final aim is to extend the (Q) property to
an infinite dimensional context. 

Recall that $T:X\rightrightarrows X^{*}$ \emph{has property (Q)}
(or is \emph{upper $\mathscr{C}-$semicontinuous}) \emph{at} $x\in X$
(with respect to the $s\times w^{*}-$topology on $X\times X^{*}$)
if for every net $\{x_{i}\}_{i\in I}\subset X$ such that $x_{i}\rightarrow x$,
strongly in $X$ we have\[
\bigcap_{i\in I}\operatorname*{cl}\!\,_{w^{*}}(\operatorname*{conv}\bigcup_{j\ge i}Tx_{j})\subset Tx.\]
Clearly, property (Q) (as well as the $s\times w^{*}-$usc property)
at $x$ has substance only when $x\in\overline{D(T)}$ and, at least
on a subnet, $\{x_{i}\}_{i}\subset D(T)$. The operator $T$ \emph{has
property (Q)} if it has property (Q) at each $x\in X$.

\begin{theorem}\label{Q} Let $X$ be a normed barrelled space and
let $T\in\mathscr{M}(X)$ be such that $\operatorname*{core}D(T)\neq\emptyset$.
Then $T$ has property (Q). \end{theorem}

\begin{proof} Let $x\in X$, let $\{x_{i}\}_{i\in I}$ be such that
$x_{i}\rightarrow x$, strongly in $X$, and let $x^{*}\in\bigcap_{i\in I}\operatorname*{cl}\!\,_{w^{*}}(\operatorname*{conv}\bigcup_{j\ge i}Tx_{j})$.

For every $i\in I$ there are $k:=k_{i}$, $I\ni j_{1},..,j_{k}\ge i$,
$x_{j_{1}}^{*}\in Tx_{j_{1}}$,.., $x_{j_{k}}^{*}\in Tx_{j_{k}}$,
$\lambda_{1},..,\lambda_{k}>0$ such that $\sum_{\ell=1}^{k}\lambda_{\ell}=1$
and $y^{*}:=\sum_{\ell=1}^{k}\lambda_{\ell}x_{j_{\ell}}^{*}$ satisfies
$|\langle x,y^{*}-x^{*}\rangle|\le1$. Pick an index $p\in\{1,..,k_{i}\}$
such that $\langle x,x_{j_{p}}^{*}\rangle\le\langle x,y^{*}\rangle$
(such an index exists since $\langle x,y^{*}\rangle=\sum_{\ell=1}^{k}\lambda_{\ell}\langle x,x_{j_{\ell}}^{*}\rangle$)
and denote $j_{p}$ by $\varphi(i)$. In this way we generate a map
$\varphi:I\rightarrow I$ such that $\varphi(i)\ge i$, for every
$i\in I$ and a net $\{(x_{\varphi(i)},x_{\varphi(i)}^{*})\}_{i\in I}\subset T$
such that $\{x_{\varphi(i)}\}_{i\in I}$ is a subnet of $\{x_{i}\}_{i\in I}$
and $\sup_{i\in I}\langle x,x_{\varphi(i)}^{*}\rangle\le\langle x,x^{*}\rangle+1<\infty$.
According to Theorem \ref{FB} (i), $x\in D(T)$, since $T$ is demiclosed.

Assume, by contradiction, that $x^{*}\not\in Tx$. Under our assumptions
we may apply Proposition \ref{si} to get $v\in\operatorname*{int}[\sigma_{Tx}<x^{*}]\subset\operatorname*{int}T_{D(T)}x$.
Let $\epsilon_{0}:=(\langle v,x^{*}\rangle-\sigma_{Tx}(v))/2>0$.
As previously seen, for every $i\in I$ there are $k:=k_{i}$, $I\ni j_{1},..,j_{k}\ge i$,
$x_{j_{1}}^{*}\in Tx_{j_{1}}$,.., $x_{j_{k}}^{*}\in Tx_{j_{k}}$,
$\lambda_{1},..,\lambda_{k}>0$ such that $\sum_{\ell=1}^{k}\lambda_{\ell}=1$
and $y^{*}:=\sum_{\ell=1}^{k}\lambda_{\ell}x_{j_{\ell}}^{*}$ satisfies
$|\langle v,y^{*}-x^{*}\rangle|\le\epsilon_{0}$. Pick an index $p\in\{1,..,k_{i}\}$
such that $\langle v,x_{j_{p}}^{*}\rangle\ge\langle v,y^{*}\rangle$
and denote $j_{p}$ by $\varphi(i)$ to generate the map $\varphi:I\rightarrow I$
with the properties $\varphi(i)\ge i$, for every $i\in I$, $\{(x_{\varphi(i)},x_{\varphi(i)}^{*})\}_{i\in I}\subset T$,
$x_{\varphi(i)}\rightarrow x$, strongly in $X$, and \begin{equation}
\inf_{i\in I}\langle v,x_{\varphi(i)}^{*}\rangle\ge\langle v,x^{*}\rangle-\epsilon_{0}.\label{fin}\end{equation}
According to Theorem \ref{FB}, there exist $\gamma>0$, $\beta\in\mathbb{R}$,
$i_{0}\in I$ such that \begin{equation}
\langle x,x_{\varphi(i)}^{*}\rangle\ge\gamma\|x_{\varphi(i)}^{*}\|+\beta\ \forall i\ge i_{0}.\label{apr}\end{equation}

If $\limsup_{i\in I}\langle x,x_{\varphi(i)}^{*}\rangle<\infty$ then,
from Theorem \ref{FB} (i), $\{x_{\varphi(i)}^{*}\}_{i\ge i'}$ is
bounded so, at least on a subnet, $x_{\varphi(i)}^{*}\rightarrow x_{0}^{*}\in Tx$,
weakly-star in $X^{*}$. Passing to limit in (\ref{fin}) leads to
the contradiction $\sigma_{Tx}(v)\ge\langle v,x_{0}^{*}\rangle\ge\langle v,x^{*}\rangle-\epsilon_{0}=\sigma_{Tx}(v)+\epsilon_{0}$.

Therefore, at least on a subnet, denoted for simplicity by the same
index, $\lim_{i\in I}\langle x,x_{\varphi(i)}^{*}\rangle=\lim_{i\in I}\|x_{\varphi(i)}^{*}\|=\infty$.
From Theorem \ref{FB} (ii) we know that $\|x_{\varphi(i)}^{*}\|^{-1}x_{\varphi(i)}^{*}\rightarrow u^{*}\in N_{D(T)}x$,
weakly-star in $X^{*}$, and $\langle x,u^{*}\rangle\ge\gamma$. Divide
(\ref{fin}) by $\|x_{\varphi(i)}^{*}\|$ and pass to limit to obtain
$\langle v,u^{*}\rangle=0$ due to the fact that $v\in T_{D(T)}x=(N_{D(T)}x)^{-}$.

Take $\delta>0$ such that $v+\delta x\in T_{D(T)}x$. If $\limsup_{i\in I}\langle v+\delta x,x_{\varphi(i)}^{*}\rangle>-\infty$,
i.e., at least on a subnet, $\{\langle v+\delta x,x_{\varphi(i)}^{*}\rangle\}_{i}$
is bounded from below: $\langle v+\delta x,x_{\varphi(i)}^{*}\rangle\ge C>-\infty$,
for every $i\in I$. Divide again by $\|x_{\varphi(i)}^{*}\|$ and
pass to limit to find the contradiction $\delta\gamma\le\langle v+\delta x,u^{*}\rangle=0$.
Therefore $\lim_{i\in I}\langle v+\delta x,x_{\varphi(i)}^{*}\rangle=-\infty$
from which we get, taking (\ref{fin}) into account, that $\lim_{i\in I}\langle x,x_{\varphi(i)}^{*}\rangle=-\infty$,
which contradicts (\ref{apr}). The proof is complete. \end{proof}

\begin{theorem}\label{Q-ic} Let $X$ be a Banach space and let $T\in\mathscr{M}(X)$
be such that $^{ic}D(T)\neq\emptyset$. Then $T$ has property (Q).
\end{theorem}

\begin{proof}As usual assume that $0\in F:=\operatorname*{aff}D(T)$.
Since $T\in\mathscr{M}(X)$ we know that $T_{F}\in\mathscr{M}(F)$,
$\operatorname*{int}D(T_{F})=\operatorname*{ri}D(T)={}^{ic}D(T)\neq\emptyset$,
and $T=T+N_{F}$. According to Theorem \ref{Q}, $T_{F}$ has property
(Q) with respect to the $s\times w^{*}-$topology on $F\times F^{*}$. 

Let $x\in\overline{D(T)}$, let $\{x_{i}\}_{i\in I}\subset D(T)$
be such that $x_{i}\rightarrow x$, strongly in $X$, and let $x^{*}\in\bigcap_{i\in I}\operatorname*{cl}\!\,_{w^{*}}(\operatorname*{conv}\bigcup_{j\ge i}Tx_{j})$.
Then $x^{*}|_{F}\in\bigcap_{i\in I}\operatorname*{cl}\!\,_{w^{*}}(\operatorname*{conv}\bigcup_{j\ge i}T_{F}x_{j})$
due to the continuity of $\iota^{*}:(X^{*},w^{*})\rightarrow(F^{*},w^{*})$,
$\iota^{*}(x^{*})=x^{*}|_{F}$. Hence $x^{*}|_{F}\in T_{F}x$, that
is $x^{*}\in Tx+F^{\perp}=Tx$. \end{proof}

\bibliographystyle{plain}

\begin{thebibliography}{10}

\bibitem{MR749753}
Jean-Pierre Aubin and Ivar Ekeland.
\newblock {\em Applied nonlinear analysis}.
\newblock Pure and Applied Mathematics (New York). John Wiley \& Sons Inc., New
  York, 1984.
\newblock A Wiley-Interscience Publication.

\bibitem{MR995141}
Jon Borwein and Simon Fitzpatrick.
\newblock Local boundedness of monotone operators under minimal hypotheses.
\newblock {\em Bull. Austral. Math. Soc.}, 39(3):439--441, 1989.

\bibitem{MR2362689}
Jonathan~M. Borwein.
\newblock Asplund decompositions of monotone operators.
\newblock In {\em C{SVAA} 2004---control set-valued analysis and applications},
  volume~17 of {\em ESAIM Proc.}, pages 19--25. EDP Sci., Les Ulis, 2007.

\bibitem{MR0180884}
Felix~E. Browder.
\newblock Multi-valued monotone nonlinear mappings and duality mappings in
  {B}anach spaces.
\newblock {\em Trans. Amer. Math. Soc.}, 118:338--351, 1965.

\bibitem{MR1974634}
Regina~Sandra Burachik and B.~F. Svaiter.
\newblock Maximal monotonicity, conjugation and the duality product.
\newblock {\em Proc. Amer. Math. Soc.}, 131(8):2379--2383 (electronic), 2003.

\bibitem{MR1079061}
Ioana Cioranescu.
\newblock {\em Geometry of {B}anach spaces, duality mappings and nonlinear
  problems}, volume~62 of {\em Mathematics and its Applications}.
\newblock Kluwer Academic Publishers Group, Dordrecht, 1990.

\bibitem{MR709590}
Frank~H. Clarke.
\newblock {\em Optimization and nonsmooth analysis}.
\newblock Canadian Mathematical Society Series of Monographs and Advanced
  Texts. John Wiley \& Sons Inc., New York, 1983.
\newblock A Wiley-Interscience Publication.

\bibitem{MR1191009}
S.~Fitzpatrick and R.~R. Phelps.
\newblock Bounded approximants to monotone operators on {B}anach spaces.
\newblock {\em Ann. Inst. H. Poincar\'e Anal. Non Lin\'eaire}, 9(5):573--595,
  1992.

\bibitem{MR0425687}
S.~P. Fitzpatrick.
\newblock Continuity of nonlinear monotone operators.
\newblock {\em Proc. Amer. Math. Soc.}, 62(1):111--116 (1977), 1976.

\bibitem{MR1009594}
Simon Fitzpatrick.
\newblock Representing monotone operators by convex functions.
\newblock In {\em Workshop/{M}iniconference on {F}unctional {A}nalysis and
  {O}ptimization ({C}anberra, 1988)}, volume~20 of {\em Proc. Centre Math.
  Anal. Austral. Nat. Univ.}, pages 59--65. Austral. Nat. Univ., Canberra,
  1988.

\bibitem{MR0410335}
Richard~B. Holmes.
\newblock {\em Geometric functional analysis and its applications}.
\newblock Springer-Verlag, New York, 1975.
\newblock Graduate Texts in Mathematics, No. 24.

\bibitem{MR687963}
Shui~Hung Hou.
\newblock On property {$({\rm Q})$} and other semicontinuity properties of
  multifunctions.
\newblock {\em Pacific J. Math.}, 103(1):39--56, 1982.

\bibitem{MR2465513}
Andreas L{\"o}hne.
\newblock A characterization of maximal monotone operators.
\newblock {\em Set-Valued Anal.}, 16(5-6):693--700, 2008.

\bibitem{MR2675662}
M.~Marques~Alves and B.~F. Svaiter.
\newblock Maximal monotonicity, conjugation and the duality product in
  non-reflexive {B}anach spaces.
\newblock {\em J. Convex Anal.}, 17(2):553--563, 2010.

\bibitem{MR0132379}
George~J. Minty.
\newblock On the maximal domain of a ``monotone'' function.
\newblock {\em Michigan Math. J.}, 8:135--137, 1961.

\bibitem{MR984602}
Robert~R. Phelps.
\newblock {\em Convex functions, monotone operators and differentiability},
  volume 1364 of {\em Lecture Notes in Mathematics}.
\newblock Springer-Verlag, Berlin, 1989.

\bibitem{MR1238715}
Robert~R. Phelps.
\newblock {\em Convex functions, monotone operators and differentiability},
  volume 1364 of {\em Lecture Notes in Mathematics}.
\newblock Springer-Verlag, Berlin, second edition, 1993.

\bibitem{MR0253014}
R.~T. Rockafellar.
\newblock Local boundedness of nonlinear, monotone operators.
\newblock {\em Michigan Math. J.}, 16:397--407, 1969.

\bibitem{MR1451876}
R.~Tyrrell Rockafellar.
\newblock {\em Convex analysis}.
\newblock Princeton Landmarks in Mathematics. Princeton University Press,
  Princeton, NJ, 1997.
\newblock Reprint of the 1970 original, Princeton Paperbacks.

\bibitem{MR1491362}
R.~Tyrrell Rockafellar and Roger J.-B. Wets.
\newblock {\em Variational analysis}, volume 317 of {\em Grundlehren der
  Mathematischen Wissenschaften [Fundamental Principles of Mathematical
  Sciences]}.
\newblock Springer-Verlag, Berlin, 1998.

\bibitem{MR1723737}
Stephen Simons.
\newblock {\em Minimax and monotonicity}, volume 1693 of {\em Lecture Notes in
  Mathematics}.
\newblock Springer-Verlag, Berlin, 1998.

\bibitem{MR2207807}
M.~D. Voisei.
\newblock A maximality theorem for the sum of maximal monotone operators in
  non-reflexive {B}anach spaces.
\newblock {\em Math. Sci. Res. J.}, 10(2):36--41, 2006.

\bibitem{MR2389004}
M.~D. Voisei.
\newblock Calculus rules for maximal monotone operators in general {B}anach
  spaces.
\newblock {\em J. Convex Anal.}, 15(1):73--85, 2008.

\bibitem{MR2453098}
M.~D. Voisei.
\newblock The sum and chain rules for maximal monotone operators.
\newblock {\em Set-Valued Anal.}, 16(4):461--476, 2008.

\bibitem{astfnc}
M.~D. Voisei.
\newblock A sum theorem for ({FPV}) operators and normal cones.
\newblock {\em J. Math. Anal. Appl.}, 371:661--664, 2010.

\bibitem{MR2594359}
M.~D. Voisei and C.~Z{\u{a}}linescu.
\newblock Linear monotone subspaces of locally convex spaces.
\newblock {\em Set-Valued Var. Anal.}, 18(1):29--55, 2010.

\bibitem{MR2577332}
M.~D. Voisei and C.~Z{\u{a}}linescu.
\newblock Maximal monotonicity criteria for the composition and the sum under
  weak interiority conditions.
\newblock {\em Math. Program.}, 123(1, Ser. B):265--283, 2010.

\bibitem{MR1921556}
C.~Z{\u{a}}linescu.
\newblock {\em Convex analysis in general vector spaces}.
\newblock World Scientific Publishing Co. Inc., River Edge, NJ, 2002.

\end{thebibliography}

\end{document}